\documentclass[11pt]{amsart}
\usepackage{palatino}
\usepackage{latexsym, amssymb, amsfonts, amsmath, amsthm, graphics, graphicx, subfigure, bbm, stmaryrd}
\usepackage{epstopdf}
\usepackage[pdftex,colorlinks,linkcolor=blue,menucolor=red,backref=false,bookmarks=true]{hyperref}

\newcommand{\field}[1]{\mathbb{#1}}
\newcommand{\R}{\field{R}}
\newcommand{\N}{\field{N}}

\newcommand{\Tr}{\mathrm{Tr}}

\newcommand{\Var}{\mathrm{Var}}
\newcommand{\supp}{\mathrm{supp}}
\newtheorem{theorem}{Theorem}

\newtheorem{proposition}{Proposition}
\newtheorem{definition}{Definition}
\newtheorem{lemma}{Lemma}
\newtheorem{corollary}{Corollary}
\newtheorem{remark}{Remark}
\newtheorem{assumption}{Assumption}

\numberwithin{equation}{section}

\date{}

\begin{document}

\title{\bf The One Dimensional Free Poincar\'e Inequality}

\author{Michel Ledoux}
\address{ Institut de Math\'ematiques de Toulouse, Universit\'e de Toulouse, F-31062 Toulouse, France,
and Institut Universitaire de France }
\email{ledoux@math.univ-toulouse.fr}

\author{Ionel Popescu} 
\thanks{I.P. was partially supported by Marie Curie Action grant nr. 249200.}
\address{Georgia Institute of Technology, 686 Cherry Street, Atlanta GA, 30332, USA} 
\address{ Institute of Mathematics of Romanian Academy,  21 Calea Grivitei Street,  010702-Bucharest, Sector 1, Romania} 
\email{ipopescu@math.gatech.edu}
\email{ionel.popescu@imar.ro}

\begin{abstract} In this paper we discuss the natural candidate for the one dimensional free Poincar\'e inequality.  Two main strong points sustain this candidacy.  One is the random matrix heuristic and the other the relations with the other free functional inequalities,  namely, the free transportation and Log-Sobolev
inequalities.   As in the classical case the Poincar\'e is implied by the others.  
This investigation is driven by a nice lemma of Haagerup which relates logarithmic potentials and Chebyshev polynomials.   The Poincar\'e inequality revolves around the counting number operator for the Chebyshev polynomials of first kind with respect to the arcsine law on $[-2,2]$.   This counting number operator appears naturally in  a representation of the minimum of the logarithmic energy with external fields discovered in \cite{GP} as well  as in the perturbation of logarithmic energy with external fields, which is the essential connection between all these inequalities.   
 \end{abstract}

\maketitle

Classically, Poincar\'e's inequality for a probability measure $\mu$ on $\R^{d}$ states that there is a constant $\rho>0$, such that  for any compactly supported smooth function $f$, 
\begin{equation}\label{ei:0}
\rho\Var_{\mu}(f)\le \int |\nabla f|^{2}d\mu,
\end{equation}
with the notation $\Var_{\mu}(f)=\int f^{2}\,d\mu-(\int f\,d\mu)^{2}$.
This is in fact a statement about the spectral gap of the operator $L$, whose Dirichlet form is $\Gamma(f,f)=\int |\nabla f|^{2}\,d\mu$ (and invariant measure $\mu $).
This inequality is actually one member of a family of functional inequalities which are connected by implications among them.   For example, among others, the transportation and Log-Sobolev inequalities always imply the Poincar\'e with the same constant (see e.g. \cite{Ba, OV, BL2, Vi}).  

With the boom in the interest of large dimensional phenomena, one natural question is to ask what happens with the functional inequalities in the limit.  This was studied in various forms for various measures in infinite dimensions, as for example the Wiener measures with a few samples \cite{Gross}, \cite{FU2}, \cite{Wang1}, \cite{Fang1}, \cite{Hsu1}.  The important part in dealing with these infinite dimensional objects was due to the dimension independent constants in the finite dimensional approximations.  

Important  interesting limiting objects are obtained in free probability by considering random matrices.  It is well known that properly normalized, the eigenvalue distribution
the Gaussian Unitary Ensemble converges (in mean and almost surely) to the semicircular
law.  On the other hand, applying classical functional inequalities to the distribution of random matrices in dimension $n$ and taking their limits, one obtains various functional inequalities for the semicircular.   This was done in a more general situation for the Log-Sobolev by Biane \cite{Biane2} and the transportation inequality by Hiai, Petz and Ueda \cite{HPU1} and \cite{L}.    

The interesting part of this limiting procedure is that the obtained functional inequalities in the framework of free probability have a life of their own.  As such, the goal is then
to understand them from a perspective which does not appeal to their finite dimensional approximation.   There are indeed some cases where the approximation seems hard or very unnatural.   This was the main theme of the paper \cite{LP} where several techniques from mass transportation were introduced to deal with the (one dimensional) free Log-Sobolev, transportation, HWI and Brunn-Minkowski inequalities.  

Using random matrix heuristics one can add to the family already described a new member. 
This is a free Poincar\'e inequality,  the natural limit of the Poincar\'e inequality applied to random matrices, which was discussed in \cite{LP}.  The statement for such an inequality in the case of the semicircular law $\alpha(dx)=\mathbbm{1}_{[-2,2]}(x)\frac{\sqrt{4-x^{2}}dx}{2\pi}$ is that for any smooth function $f$ on the interval $[-2,2]$,
\begin{equation}\label{ei:2}
\iint\left( \frac{f(x)-f(y)}{x-y}  \right)^{2}\omega(dx\,dy)\le \int (f' )^{2}\,d\alpha
\end{equation}
where 
\[
\omega(dx\,dy)=\mathbbm{1}_{[-2,2]}(x)\mathbbm{1}_{[-2,2]}(y)\frac{(4-xy)dxdy}{4\pi^{2}\sqrt{(4-x^{2})(4-y^{2})}} \, .
\]
Notice here that this statement has a different flavor as its classical counterpart.  In the case of  the standard Gaussian measure for example, inequality \eqref{ei:0} is the expression of the spectral gap of the Ornstein-Uhlenbeck operator.  In the free case, it was shown in \cite{LP} that \eqref{ei:2} is equivalent to 
\[
\mathcal{N}\le \mathcal{L}
\]
where $(\mathcal{L}f)(x)=-(4-x^{2})f''(x)+xf'(x)$ and $\mathcal{N}$
are respectively the Jacobi operator and the counting number operator for the orthonormal basis of Chebyshev polynomials $T_{n}(x/2)$ of $L^{2}(\beta)$, where $\beta$ is the arcsine law $\mathbbm{1}_{[-2,2]}(dx)\frac{dx}{\pi\sqrt{4-x^{2}}}$. 
At least at a first look,  we are not comparing a second order operator with a projection,
as in the classical case,  but with an integro-dfferential operator.   However, for this particular case, it is true that $\mathcal{L}=\mathcal{N}^{2}$, and thus the above comparison is essentially the spectral gap for $\mathcal{N}$.      

Another natural interpretation of \eqref{ei:2} is that the $L^{2}$ norm of the classical derivative $f'$ with respect to $\alpha$ is greater than the $L^{2}$ norm
of the non-commutative derivative $Df=\frac{f(x)-f(y)}{x-y}$ with respect to a suitable measure $\omega$. This non-commutative derivative is very natural in free probability theory and it is not the first time it appears in some form of Poincar\'e's inequality.  In fact, Biane in \cite{Biane2} sets up a Poincar\'e inequality in several non-commuting variables which in the one dimensional case amounts to 
\begin{equation}\label{eq:pb}
\Var_{\alpha}(f)\le \iint \left( \frac{f(x)-f(y)}{x-y}\right)^{2}\alpha(dx)\alpha(dy)
\end{equation}
for any $C^{1}$ function $f$ on $[-2,2]$.   This is more in the classical spirit with the role of the derivative played by the non-commutative derivative $\frac{f(x)-f(y)}{x-y}$.  At any rate this inequality can be translated into the spectral gap for the counting number operator $\mathcal{M}$ associated to the scaled Chebyshev polynomials of the second kind for the semicircular law $\alpha $. This fact
makes \eqref{ei:2} and \eqref{eq:pb} formally the same.  However, this argument does not show a more structural tie between the two versions of Poincar\'e's. A more organic appearance of the counting number operator $\mathcal{M}$ in the life of \eqref{ei:2} is revealed in  Section~\ref{s:4}.  However, the only spectral properties of $\mathcal{M}$ which contributes is the mere non-negativity.  

The main investigation of this work is actually to demonstrate that
the operator point of view emphasized towards the description of the Poincar\'e inequalities
\eqref{ei:2} and \eqref{eq:pb} and their relationships may be pushed forward to
similarly study Poincar\'e inequalities for large classes of equilibrium measures and
not only the semicircular law. In particular, this analysis reveals the suitable Poincar\'e inequality
in the free context, and allows for the connections with the other functional inequalities.

It is well known that  \eqref{ei:0} is valid for measures $\mu(dx)=e^{-V(x)}dx$, where $V$ is strong convex.   In fact, if $V$ is strong convex, say $V''(x)\ge\rho$ for some $\rho >0 $ and all $x\in\R$,
applying Poincar\'e's inequality \eqref{ei:0} to the measure $e^{-n\Tr V(X)}dX$
on Hermitian $n\times n$ matrices and functions of the form $\Phi(X)=\Tr(\phi(X))$ (for details see \cite{LP}) leads to: 
\begin{equation}\label{ei:3}
 2\rho c^{2} \iint\left( \frac{f(x)-f(y)}{x-y}  \right)^{2}\omega_{b,c}(dx\,dy)\le\int (f')^{2}d\mu_{V}
\end{equation}
for any $C^{1}$ function on the support  $[-2c+b,2c+b]$ of $\mu_{V}$.  Here $\mu_{V}$ is the equilibrium measure (i.e. the minimizer) of 
\begin{equation}\label{ei:7}
E_{V}(\mu)=\int V\,d\mu-\iint \log|x-y|\mu(dx)\mu(dy)
\end{equation}
over the sef of all probability measures on $\R$.   It is well known (see for example \cite{ST}) that the support of $\mu_{V}$ is one interval in the case $V$ is convex.  The measure  $\omega_{b,c}$
on the left hand side of
\eqref {ei:3} is just a linear rescaling of the measure $\omega$ defined above, precisely 
\begin{equation}\label{ei:4}
\omega_{b,c}(dx\,dy)=\mathbbm{1}_{[-2c+b,2c+b]}(x)\mathbbm{1}_{[-2c+b,2c+b]}(y)\frac{(4c^{2}-(x-b)(y-b))dxdy}{4c^{2}\pi^{2}\sqrt{(4c^{2}-(x-b)^{2})(4c^{2}-(y-b)^{2})}} \, .
\end{equation}
The point is now, \eqref{ei:3} is a well-defined notion on its own for any given
probability measure $\mu$ on the interval $[-2c+b,2c+b]$. It defines the canonical free
Poincar\'e inequality which will be investigated in this
work. As it is in the case of \eqref{ei:2} for the semicircular law, this inequality gravitates around the counting number operator $\mathcal{N}$. The investigation is driven by a lemma by Haagerup
which was extensively used in \cite{GP} to deal with the minimization of the logarithmic
energy with external fields) providing an analytic description of the number operator $\mathcal{N}$ as
$$  (\mathcal{N}\phi)(x) =\int y\phi'(y)\beta(dy)  +x\int \phi'(y)\beta(dy)
       - (4-x^{2})\int \frac{\phi'(x)-\phi'(y)}{x-y} \, \beta(dy) $$     
which connects with free derivatives. In particular, this description produces concise
and efficient interpretations of the equilibrium measure $\mu_V$ and logarithmic energy $E_V$
associated to an external field $V$ of independent interest. With these tools, the
free Poincar\'e inequality for a measure $\mu $ may then be described at the operator
level as the comparison of $\mathcal{N}$ with the operator with Dirichlet form
$\int (f')^2 d\mu$. 

In \cite{LP}, another version of Poincar\'e's inequality inspired by Biane's version of \eqref{eq:pb}
was presented,
which states that for $\mu$ with compact support, there is a constant $\rho>0$ such that  
\begin{equation}\label{eq:pb2}
\rho\Var_{\mu}(f)\le \iint \left( \frac{f(x)-f(y)}{x-y}\right)^{2}\mu(dx)\mu(dy)
\end{equation}
as long as $f$ is $C^{1}(\R)$. Besides the example of the semicircular law, we  
do not know however if there is any (interesting) connection between
the two Poincar\'e inequalities \eqref{ei:3} and \eqref{eq:pb2}. 
As we will show, the Poincar\'e inequality \eqref{ei:3} will be justified by its
connection with the transportation and  Log-Sobolev inequalities
(which does not seem of the same nature for \eqref{eq:pb2}).

Indeed, once the proper free Poincar\'e inequality \eqref{ei:3} has been identified, the next purpose is to investigate its relationships with the traditional free functional inequalities such as
transportation and Log-Sobolev inequalities. The free transportation inequality associated to a potential $V$ claims that there is a $\rho>0$ with the property that 
\begin{equation}\label{ei:5}
\rho W_{2}^{2}(\nu,\mu_{V})\le E_{V}(\nu)-E_{V}(\mu_{V})
\end{equation}
for any other probability measure $\nu$ on the real line.  Free  Log-Sobolev states that there is a $\rho>0$ so that for any other (sufficiently nice) probability measure $\nu$,
\begin{equation}\label{ei:6}
E_{V}(\nu)-E_{V}(\mu_{V})\le \frac{1}{4\rho}\int (H\nu-V')^{2}d\nu
\end{equation}
where $H\nu=p.v. \int \frac{2}{x-y}\nu(dy)$ is the Hilbert transform of the measure $\nu$.   
In this paper we show that under some mild assumptions, the transportation and Log-Sobolev inequalities
imply the free Poincar\'e inequality \eqref {ei:3}. It should be pointed out that these implications are easy or standard
in the classical case. That the Poincar\'e inequality follows from the Log-Sobolev is obtained
by a simple Taylor expansion on (classical) entropy (see e.g. \cite {Ba, Vi}). The implication from the transportation
inequality is a bit more involved, the simplest argument going through Hamilton-Jacobi equations
(\cite {BL2, Vi}).
Actually, what the classical case puts forward is the necessity of suitable perturbation
properties of both the logarithmic energy and equilibrium measure in the free context.
This will be achieved in the second part of this paper.
At the heart of the argument is a perturbation argument for the logarithmic energy
$E_{V}$, which is given by the counting number operator $\mathcal{N}$, the same
one which plays the key role in understanding the free Poincar\'e inequality. Again, this perturbation
property might be of independent interest.   
 
Here is how the paper is organized.  In Section~\ref{s:2} we introduce the preliminary material, namely the logarithmic potentials, Chebyshev polynomials and we briefly discuss Haagerup's Lemma. We
also introduce and study several related operators, the most important one being the counting number
operator $\mathcal{N}$ and its analytic description.

Section~\ref{s:3} is the one introducing the Poincar\'e inequality 
\eqref {ei:3} and several associated properties, while Section~\ref{s:4}
investigates various equivalent characterizations of this.  The main ones are equivalent via some sort of duality, which is somewhat reminiscent of the duality associated to the Monge-Kantorovich distance
in the theory of mass transportation.     

In Section~\ref{s:5} we give the perturbation results which is the backbone for
the connection between the other functional inequalities and Poincar\'e.  This last
connection is discussed in Section~\ref{s:fineq} together with a discussion about why the perturbation used in the classical case to go from the Log-Sobolev and transportation is not enough.

\section{Preliminaries}\label{s:2}

In this section we introduce some basic notions we are going to use in this paper.

A \emph{potential} on a  closed subset $S$ of the real line is simply a function $V:S\to\R$.  For our investigation of the logarithmic potentials with external fields, we will assume that $V$ is of class $C^{3}$ on the interior of $S$ and that if $S$ is unbounded, 
\[
\lim_{|x|\to\infty} V(x)-2\log|x|=+\infty.  
\]
We will call such a potential admissible.    

For a probability measure $\mu$ the \emph{logarithmic energy with external field $V$} is given by 
\begin{equation}\label{ep:lp}
E_{V}(\mu)=\int Vd\mu-\iint \log|x-y|\mu(dx)\mu(dy).
\end{equation}
It is known that given a closed subset $S$ and an admissible potential $V$ (see \cite{ST} or \cite{Deift1})  there is a unique minimizer $\mu_{V}$ in the set of probability measures on $S$.  In addition this measure also has compact support.   We will denote for simplicity $E_{V}=E_{V}(\mu_{V})$.  
The support of the measure $\mu$ is denoted by $\supp{\mu}$.

The equilibrium measure $\mu_{V}$ of \eqref{ep:lp} on the set $S$ (cf. 
\cite[Thm.I.1.3]{ST}) is characterized by 
\begin{equation}\label{ep:var}
\begin{split}
V(x)&\ge 2\int \log|x-y|\mu(dy)+C \quad\text{quasi-everywhere on }S \\ 
V(x)&= 2\int \log|x-y|\mu(dy)+C 
\quad\text{quasi-everywhere on}\:\:\supp{\mu}. \\ 
\end{split}
\end{equation} 
For the definition of the notion of quasi-everywhere, we refer the reader to \cite{ST}.

What we will need from this is in particular that the equality on $\supp \mu$ is 
almost surely with respect to any probability measure of finite logarithmic energy.  

If $(X,\mathcal{X})$, $(Y,\mathcal{Y})$ are two measurable spaces,  $\mu$ is a measure on $X$ and $\phi:X\to Y$, is a measurable map,  we set $\phi_{\#}\mu$ to stand for the push forward measure
\[
(\phi_{\#}\mu)(A)=\mu(\phi^{-1}(A))
\]
for any $A\in\mathcal{Y}$.

It is easy to verify that changing the variable of integration to $x\to c x+b$ 
and $y\to c y+b$, with $c\ne0$, setting $\ell_{b,c}(x)=(x-b)/c$ and 
$ \mu_{c,b}=(\ell_{b,c})_{\#}\mu$, the following holds
\begin{equation}\label{ep:4}
E_{V}(\mu)=\int V(cx+b)\mu_{b,c}(dx)-\iint\log|cx-cy|\mu_{b,c}(dx)
\mu_{b,c}(dy)= E_{V(\ell_{b,c}^{-1})-\log(c)}(\mu_{b,c})
=E_{V(\ell_{b,c}^{-1})}(\mu_{b,c})-\log c
\end{equation}
which in turn results with 
\begin{equation}\label{ep:5}
\notag E_{V}=E_{V(\ell_{b,c}^{-1})-\log(c)}= E_{V(\ell^{-1}_{b,c})}-\log(c).
\end{equation}

\subsection{Connection with Chebyshev Polynomials}

Recall that the {\em Chebychev polynomials of the first kind}  $T_n(x)$ are defined by
\begin{equation}\label{ep:T}
T_n(\cos\theta)=\cos(n\theta).
\end{equation}
 Alternatively, they are given by the recursion
relation
\[
T_{n+1}(x)=2xT_n(x)-T_{n-1}(x), \qquad T_0(x)=1, \,\, T_1(x)=x
\]
with the generating function
\begin{equation}\label{ep:gfT}
\sum_{n = 0}^\infty r^{n}T_{n}(x)=\frac{1-rx}{1-2rx+r^{2}} \, , \quad  r\in (-1,1).
\end{equation}
 If we take $\tilde{T}_{0}=T_{0}$ and $\tilde{T}_{n}(x)=\sqrt{2}T_{n}(x)$, then $\{\tilde{T}_{n}\}_{n\ge0}$ is the sequence of orthogonal polynomials for the {\em arcsine law}  $\mathbbm{1}_{(-1,1)}(x)\frac{dx}{\pi\sqrt{1-x^{2}}}$. 

The \emph{Chebyshev polynomials of second kind} $U_{n}$ are defined by 
\begin{equation}\label{ep:U}
U_{n}(\cos \theta)=\frac{\sin (n+1)\theta}{\sin\theta}.
\end{equation}
These satisfy the recurrence 
\[
U_{n+1}(x)=2xU_{n}(x)-U_{n-1}(x), \qquad  U_{0}=1, \:U_{1}=2x
\]
and the generating function is 
\begin{equation}\label{ep:gfU}
\sum_{n = 0} ^\infty r^{n}U_{n}(x)=\frac{1}{1-2rx+r^{2}} \, , \quad  r\in(-1,1).
\end{equation}

These are the orthogonal polynomials for the {\em semicircular distribution}
$\mathbbm{1}_{[-1,1]}(x)\frac{2\sqrt{1-x^{2}}dx}{\pi}$. 

The main connection between the Chebyshev polynomials of the first and second kind is given by
\begin{equation}\label{ep:TU}
T_{n}'(x)=nU_{n-1}(x).  
\end{equation}

In the sequel we will use the following notation 
\[
\phi_{n}(x)=T_{n}\left(\frac{x}{2} \right) \quad
    \text{ and } \quad \psi_{n}(x)=U_{n}\left(\frac{x}{2} \right) \quad \text{ for }n\ge0.
\]
We mention that these are the orthogonal polynomials for the arcsine and semicircular on $[-2,2]$.  

It is easy to check the following relations between $\phi_{n}$ and $\psi_{n}$:
\begin{equation}\label{eq:phipsi}
\begin{split}
2\phi_{n}(x)\phi_{m}(x)&=\phi_{|n-m|}(x)+\phi_{n+m}(x), \quad n,m\ge0\\
2\psi_{n}(x)\phi_{m}(x)&={\mathrm{sign}}(n+1-m)\psi_{|n+1-m|-1}(x)+\psi_{n+m}(x), \quad n,m\ge0\\
\frac{(4-x^{2})}{2} \, \psi_{n}(x)\psi_{m}(x)&=\phi_{|n-m|}(x)-\phi_{n+m+2}(x),\quad n,m\ge0,\\
\end{split}
\end{equation}
where here and throughout this paper, ${\mathrm{sign}}(x)=1$ for $x>0$, ${\mathrm{sign}}(x)=-1$ for $x<0$ and ${\mathrm{sign}}(x)=0$ for $x=0$.  

The following Lemma which will play an important role in the subsequent analysis 
appears in some seminar notes of Haagerup \cite{Ha}.

\begin{lemma}[Haagerup]\label{l:1}
\begin{enumerate}
 
\item For any real $x,y\in [-2,2]$, $x\ne y$, we have
\begin{equation}
\log|x-y|=-\sum_{n=1}^{\infty}\frac{2}{n}\phi_{n}(x)\phi_{n}(y) \label{ep:1}
\end{equation}
where the series here are convergent on $x\ne y$.  
\item For $x>2$ and  $y\in[-2,2]$, a similar expansion takes place, 
\[
\log|x-y|=\log\left|\frac{x+\sqrt{x^{2}-4}}{2} \right|-\sum_{n=1}^{\infty}
\frac{2}{n}\left( \frac{x-\sqrt{x^{2}-4}}{2} \right)^{n}\phi_{n}(y)
\]
where the series is absolutely convergent.  

\item The logarithmic potential of a 
probability  measure $\mu $ on $[-2,2]$ is given by 
\begin{equation}\label{ep:2}
\int \log|x-y|\mu(dx)= - \sum_{n=1}^{\infty}\frac{2}{n} \, \phi_{n}(x)
\int \phi_{n}(y)\mu(dy)
\end{equation}
where this series makes sense pointwise.  Therefore, the logarithmic energy of the 
measure $\mu$ is given by
\begin{equation}\label{ep:3}
\iint \log|x-y|\mu(dx)\mu(dy) = -\sum_{n=1}^{\infty}\frac{2}{n}\left(
\int \phi_{n}(x)\mu(dx)\right)^{2}.
\end{equation}
In particular $\iint \log|x-y|\mu(dx)\mu(dy)$ is finite if and only if 
$\sum_{n=1}^{\infty}\frac{2}{n}\left(\int \phi_{n}(x)
\mu(dx)\right)^{2}$ is finite.
\end{enumerate}
\end{lemma}  

\begin{proof}
A full scale proof is given in \cite{GP}, here we only outline the main calculation leading to \eqref{ep:1}.  
Write  $x=2\cos u$ and $y=2\cos v$, and observe 
\[
x-y=2(\cos u-\cos v)=4\sin\left(\frac{u+v}{2}\right)\sin\left(\frac{u-v}{2}
\right).
\]
Hence
\begin{align*}
\log|x-y|& =\log \left|2\sin\left(\frac{u+v}{2} \right) \right| 
+\log \left|2\sin\left(\frac{u-v}{2} \right) \right| \\ 
&=\log|1-e^{i(u+v)}|+\log|1-e^{i(u-v)}|\\ 
& = {\rm Re} \left(\log(1-e^{i(u+v)})+\log(1-e^{i(u-v)}) \right) \\ 
& = -\sum_{n=1}^{\infty}\frac{1}{n} \,  {\rm Re}\left( e^{in(u+v)}+e^{in(u-v)} \right) \\ 
& = -\sum_{n=1}^{\infty}\frac{1}{n} \left(\cos(n(u+v))+\cos(n(u-v)) \right) \\ 
& = -\sum_{n=1}^{\infty}\frac{2}{n} \, \cos (nu)\cos(nv) \\
& = -\sum_{n=1}^{\infty}\frac{2}{n} \, \phi_{n}(x)\phi_{n}(y).
\end{align*}
Notice that in the middle of this we used the fact that for a complex number $z$, with $|z|=1$, $z\ne 1$, the usual logarithmic formula which computes the logarithm is still valid: 
\[
\log(1-z)=-\sum_{k\ge1}\frac{z^{k}}{k} \, .  
\]\qedhere  
\end{proof}

It is this simple lemma which gives the theme of dealing with logarithmic energies of measures by reducing them via rescaling to measures on the interval $[-2,2]$. The next statement is a simple consequence.
 
\begin{corollary}\label{c:1}
 If $\beta(dx)=\mathbbm{1}_{[-2,2]}(x)\frac{dx}{\pi\sqrt{4-x^{2}}}$ is the 
arcsine law of the interval $[-2,2]$, then 
\begin{equation}\label{ep:6}
\int \log|x-y|\,\beta(dy)=\begin{cases} 
0,& |x|\le2\\
\log\frac{|x|+\sqrt{x^{2}-4}}{2} ,& |x|>2.
\end{cases}
\end{equation}
If $\mu$ is a signed measure on $[-2,2]$ with finite total variation and 
finite logarithmic energy, then 
\begin{equation}\label{ep:7}
\int \log|x-y|\mu(dy)=c \, \, \, 
\text{ almost everywhere for all }  \,  x\in[-2,2] 
\end{equation}
 if and only if $\mu(dx)=\beta(dx)$.  Here,  ``almost everywhere'' is understood with respect 
to the Lebesgue measure.  Additionally,  the constant $c$ must be $0$.  
\end{corollary}

We define the following probability measures related to the interval $[-2c+b,2c+b]$ which are used throughout this note.  
\begin{equation}\label{e:abo}
\begin{split}
\alpha_{b,c}(dx)&=\mathbbm{1}_{[-2c+b,2c+b]}(x)\frac{\sqrt{4c^{2}-(x-b)^{2}}dx}{2\pi c^{2}}  \\
\beta_{b,c}(dx)&=\mathbbm{1}_{[-2c+b,2c+b]}(x)\frac{dx}{\pi \sqrt{4c^{2}-(x-b)^{2}}}\\
\omega_{b,c}(dx\,dy)&=\mathbbm{1}_{[-2c+b,2c+b]}(x)\mathbbm{1}_{[-2c+b,2c+b]}(y)\frac{(4c^{2}-(x-b)(y-b))dxdy}{4c^{2}\pi^{2}\sqrt{(4c^{2}-(x-b)^{2})(4c^{2}-(y-b)^{2})}} \, .
\end{split}
\end{equation}
We mention that $\alpha_{b,c}$, respectively $ \beta_{b,c}$, is semicircular, respectively arcsine on $[-2c+b,2c+b]$.  To be completely consistent, $\alpha_{b,c}$ is defined on the closed interval $[-2c+b,2c+b]$, while 
$\beta_{b,c}$  and $\omega_{b,c}$ are properly defined on the open sets $(-2c+b,2c+b)$ and $(-2c+b,2c+b)\times (-2c+b,2c+b)$ respectively.   On the other hand, as we will integrate functions on $[-2c+b,2c+b]$, and all these measures are absolutely continuous with respect to the Lebesgue measure on the real axis it does not matter if the integrals are on the open or closed intervals (or product of such).  Henceforth,  we set the scene for all these measures to be defined on the closed interval $[-2c+b,2c+b]$ (or $[-2c+b,2c+b]\times [-2c+b,2c+b]$ for $\omega_{b,c}$). 
  
For simplicity set $\alpha=\alpha_{0,1}$, $\beta=\beta_{0,1}$ and $\omega=\omega_{0,1}$, which are probabilities on $[-2,2]$ or $[-2,2]\times[-2,2]$.   
Notice the simple rescaling shows that  $\alpha_{b,c}=(\ell^{-1}_{b,c})_{\#}\alpha$ and similarly $\beta_{b,c}=(\ell^{-1}_{b,c})_{\#}\beta$ while  $((\ell^{2}_{b,c})^{-1})_{\#}\omega_{b,c}=\omega$ with $\ell^{2}_{b,c}:\R^{2}\to\R^{2}$, $\ell^{2}_{b,c}(x,y)=(\ell_{b,c}(x),\ell_{b,c}(y))$.

Throughout this paper we use $\langle \cdot, \cdot \rangle_{\gamma}$ to denote the scalar product in $L^{2}(\gamma)$ and reserve $\langle \cdot,\cdot \rangle$ for the inner product in $L^{2}(\beta)$.

Using Lemma~\ref{l:1} we prove the first result of this note which appears partially in \cite{GP}.  It will naturally lead to the operator formulation of the Poincar\'e inequality next.    

\begin{theorem}\label{t:1}
Assume that $V$  is a $C^{3}$ function on $[-2,2]$ and $A\in\R$ a constant.  
Then,  there is a unique signed measure $\mu$ on $[-2,2]$ of finite total 
variation which solves 
\begin{equation}\label{e:0}
\begin{cases}
2\int \log|x-y|\mu(dy)=V(x)+C  \, \, \, \text{ almost everywhere for } \, x\in[-2,2],\\
\mu([-2,2])=A
\end{cases}
\end{equation}
where almost everywhere is with respect to the Lebesgue measure on $[-2,2]$. The solution $\mu$ is given by $\mu(dx)=u(x)\beta(dx)$ where
\begin{equation}\label{e:sol}
u(x)=A-\frac{1}{2}\int_{-2}^{2}yV'(y)\beta(dy) 
-\frac{x}{2}\int_{-2}^{2}V'(y)\beta(dy)
+ \frac{4-x^{2}}{2}\int_{-2}^{2} \frac{V'(x)-V'(y)}{x-y}
\beta(dy).
\end{equation}
In addition, the constant $C$ must be given by  $C=-\int_{-2}^{2}
V(x)\,\beta(dx)$. 

Moreover, for any $C^{1}$ function $\phi$ on $[-2,2]$ we have that 
\begin{equation}\label{e:39}
\int \phi(x)\mu(dx)=A\int\phi(x)\beta(dx)
-\iint
\frac{(V(x)-V(y))(\phi(x)-\phi(y))}{(x-y)^{2}} \, \omega(dx\,dy).
\end{equation}
\end{theorem}

\begin{proof}
In the first place, the uniqueness is clear.  
To prove the rest we first write the function $V$ in terms of Chebysev polynomials of the first kind
\[
V(x)=\int V(y)\beta(dy)+2\sum_{n=1}^{\infty}\left(\int V(y)\phi_{n}(y)\beta(dy)\right)\phi_{n}(x).
\]
Assuming $(\mu, V)$ solve \eqref {e:0} and
 invoking Haagerup's representation, results now with
\[
-2\sum_{n =1}^\infty
\frac{2}{n}\left(\int \phi_{n}(y)\mu(dy)\right)\phi_{n}(x)=C+\int V(y)\beta(dy)+2\sum_{n=1}^{\infty}\left(\int V(y)\phi_{n}(y)\beta(dy)\right)\phi_{n}(x).
\]
Thus, equating the coefficients, we must have $C=-\int V(y)\beta(dy)$ and 
\[
\int \phi_{n}(x)\mu(dx) = -\frac{n}{2}\int V(x)\phi_{n}(x)\beta(dx) \quad n\ge1,
\]
which means that $\mu(dx)=u(x)\beta(dx)$ with
\[
u(x)=A-\sum_{n=1}^{\infty}n\left( \int V(y)\phi_{n}(y)\beta(dy) \right) \phi_{n}(x).
\]

To prove equality \eqref{e:sol}, our task is therefore to show that 
\begin{equation}\label{ep:tmp1}
\begin{split}
-\sum_{n=1}^{\infty}n\left( \int V(y)\phi_{n}(y)\beta(dy) \right)\phi_{n}(x)
&=-\frac{1}{2}\int_{-2}^{2}yV'(y)\beta(dy) 
-\frac{x}{2}\int_{-2}^{2}V'(y)\beta(dy)
 \\ &\quad+ \frac{4-x^{2}}{2}\int_{-2}^{2} \frac{V'(x)-V'(y)}{x-y} \, 
\beta(dy).
\end{split} 
\end{equation}
Notice that both sides of this equation are linear functions of $V$ and 
thus by a simple approximation argument it suffices to check it for the 
case of $V(x)=\phi_{n}(x)$ for some $n\ge1$, which boils down to
\begin{equation}\label{ep:20}
n\phi_{n}(x)
=\int_{-2}^{2}y\phi_{n}'(y)\,\beta(dy) 
+x\int_{-2}^{2}\phi_{n}'(y)\,\beta(dy)
-(4-x^{2})\int_{-2}^{2} \frac{\phi_{n}'(x)-\phi_{n}'(y)}{x-y}
\,\beta(dy).
\end{equation}
There are several ways of doing this.  The straightforward way is to look at the generating functions of both sides and use \eqref{ep:gfT}.  We pause now and give a more general statement which will be used later on.  

\begin{lemma}\label{l:2}  Define the operator $\mathcal{U}_{b,c}$ which for a $C^{1}$ function $f$ on $[-2c+b,2c+b]$ outputs the function $\mathcal{U}_{b,c}f$,
\[
(\mathcal{U}_{b,c}f)(x)=\int \frac{f(x)-f(y)}{x-y}  \, \beta_{b,c}(dy). 
\]
As usual, for simplicity we denote $\mathcal{U}=\mathcal{U}_{0,1}$. 
Then 
\begin{equation}\label{ep:opU}
\mathcal{U}\phi_{n}=\frac{1}{2}\psi_{n-1},\quad n\ge 1,\quad \text{ and}\quad (\mathcal{U}\psi_{n})(x)=\frac{1}{4-x^{2}}\begin{cases}  
2-2\phi_{n+1}(x), & n\text{ odd }\\
x-2\phi_{n+1}(x), & n\text{ even }
\end{cases},\quad n\ge0. 
\end{equation}
\end{lemma}

\begin{proof}

The idea is to use the generating functions \eqref{ep:gfT} and \eqref{ep:gfU} and compute the operator $\mathcal{U}$ of these generating functions.  To carry this out, let  
\begin{equation}\label{eq:GH}
g_{r}(x):=\frac{1-rx/2}{1-rx+r^{2}}\quad \text{ and } \quad  h_{r}(x):=\frac{1}{1-rx+r^{2}}
\end{equation}
which are the generating functions of $\phi_{n}$, respectively $\psi_{n}$. 
Then it is easy to compute 
\begin{equation}\label{ep:303}
(\mathcal{U}g_{r})(x)=\frac{r}{2(1-rx+r^{2})} =\frac{rh_{r}(x)}{2} \, , 
\end{equation}
which immediately implies the first half of \eqref{ep:opU}.  On the other hand 
\begin{equation}\label{ep:30}
\begin{split}
(\mathcal{U}h_{r})(x)&=\frac{r}{(1-r^{2})(1-rx+r^{2})} 
  =\frac{1}{4-x^{2}}\Bigg ( \frac{2r}{1-r^{2}}+\frac{x}{1-r^{2}}-\frac{\Big(\frac{2-rx}{1-rx+r^{2}}-2 \Big)}{r} \Bigg)\\ 
&=\frac{1}{4-x^{2}}\left(\frac{2r}{1-r^{2}}+\frac{x}{1-r^{2}}-\frac{2(g_{r}(x)-1)}{r}  \right).
\end{split}
\end{equation}
which clearly resolves the other half of \eqref{ep:opU}.\qedhere
\end{proof}

Coming back to the proof of Theorem~\ref{t:1}.
Armed with \eqref{ep:opU} and \eqref{ep:TU} and the simple fact that 
\[
\int_{-2}^{2}y\phi'_{n}(y)\beta(dy)=\begin{cases} n & (n=\text{even})\\ 0 &(n=\text{odd}) \end{cases}\quad\text{and}\quad \int_{-2}^{2}\phi'_{n}(y)\beta(dy)=\begin{cases} 0 & (n=\text{even})\\ n/2 &(n=\text{odd}) \end{cases},
\] 
it is now an easy task to verify \eqref{ep:20},
and in turn \eqref{ep:tmp1}. 
To prove equality \eqref{e:39}, we need to check that 
\begin{equation}\label{e:41}
\sum_{n=1}^{\infty}n\int V(y)\phi_{n}(y)\beta(dy)\int \phi(x)\phi_{n}(x)\beta(dx)=\frac{1}{4}
\int_{-2}^{2}\int_{-2}^{2}
\frac{(V(x)-V(y))(\phi(x)-\phi(y))}{(x-y)^{2}}
\frac{(4-xy)dxdy}{\pi^{2}\sqrt{4-x^{2}}\sqrt{4-y^{2}}} \, .
\end{equation}
To this end, notice that for $-1<r<1$, 
\begin{equation}\label{e:42}
\begin{split}
\sum_{n=1}^{\infty}nr^{n-1}\int V(y) \phi_{n}(y) &\beta(dy)  \int \phi(x)\phi_{n}(x)\beta(dx)\\
     & =\int_{-2}^{2} \int_{-2}^{2}\sum_{n=1}^{\infty}nr^{n-1}\phi_{n}(x)\phi_{n}(y)
             \frac{\phi(x)V(y)dxdy}   {\pi^{2}\sqrt{4-x^{2}} \sqrt{4-y^{2}}}\\ 
             &=-\frac{1}{2}\int_{-2}^{2}\int_{-2}^{2} \sum_{n=1}^{\infty}nr^{n-1}\phi_{n}(x)\phi_{n}(y)
       \frac{(V(x)-V(y))(\phi(x)-\phi(y))dxdy}{\pi^{2}\sqrt{4-x^{2}}\sqrt{4-y^{2}}} \, . 
\end{split}
\end{equation}
Now to compute the kernel inside the integration, notice that (here we inspire from \cite{CD1}) with 
$x=2\cos u$ and $y=2\cos v$, 
\[
\begin{split}
\sum_{n=1}^{\infty}r^{n}\phi_{n}(x)\phi_{n}(y)&=\sum_{n=1}^{\infty}r^{n}\cos(nu)\cos(nv)  
=\sum_{n=1}^{\infty}\frac{r^{n}}{2}\left(\cos(n(u+v))+\cos(n(u-v)) \right) \\ 
 &= \sum_{n=1}^{\infty}\frac{r^{n}}{4}\left(e^{in(u+v)}+e^{-in(u+v)}+e^{in(u-v)}+e^{-in(u-v)} \right)  \\
 & = \frac{1}{4}\left( \frac{1}{1-re^{i(u+v)}}+\frac{1}{1-re^{-i(u+v)}}+\frac{1}{1-re^{i(u-v)}}+\frac{1}{1-re^{-i(u-v)}} \right).
\end{split}
\]
Taking the derivative with respect to $r$, gives
\[
\sum_{n=1}^{\infty}nr^{n-1}\phi_{n}(x)\phi_{n}(y)
   =\frac{1}{4}\bigg ( \frac{e^{i(u+v)}}{(1-re^{i(u+v)})^{2}}+\frac{e^{-i(u+v)}}{(1-re^{-i(u+v)})^{2}}+\frac{e^{i(u-v)}}{(1-re^{i(u-v)})^{2}}+\frac{e^{-i(u-v)}}{(1-re^{-i(u-v)})^{2}} \bigg).
\]
Using Lebesgue's dominated convergence combined with \eqref{e:42}, 
after letting $r\uparrow 1$, the rest follows from
\[
\frac{1}{4}\bigg ( \frac{e^{i(u+v)}}{(1-e^{i(u+v)})^{2}}+\frac{e^{-i(u+v)}}{(1-e^{-i(u+v)})^{2}}+\frac{e^{i(u-v)}}{(1-e^{i(u-v)})^{2}}+\frac{e^{-i(u-v)}}{(1-e^{-i(u-v)})^{2}} \bigg)
     =-\frac{1-\cos u\cos v}{2(\cos u-\cos v)^{2}}=-\frac{4-xy}{2(x-y)^{2}} \, .\qedhere
\]
\end{proof}

Theorem \ref {t:1} motivates the introduction of the following operators.  
\begin{definition}\label{d:1} For a $C^{2}$ function $\phi:[-2,2]\to\R$, set
\begin{equation}\label{e:EN}
\begin{split}
(\mathcal{E}\phi)(x)&=-\int \log|x-y|\phi(y)\beta(dy) , \\
(\mathcal{N}\phi)(x)&=\int y\phi'(y)\beta(dy) 
+x\int \phi'(y)\beta(dy)
- (4-x^{2})\int \frac{\phi'(x)-\phi'(y)}{x-y} \,  \beta(dy).
\end{split}
\end{equation}
\end{definition}

Using the above theorem it is clear that $\mathcal{N}\phi$ is the unique solution $\psi$ which satisfies 
\[
\begin{cases}
\int \log|x-y|\psi(y) \beta(dy)=-\phi(x)+\int \phi\, d\beta\quad \text{ almost everywhere for } x\in[-2,2],\\
\int \psi\, d\beta=0,
\end{cases}
\]
where almost everywhere is with respect to the Lebesgue measure on $[-2,2]$.  

We collect the main properties of the operators $\mathcal{E}$ and $\mathcal{N}$ in the following. 

\begin{proposition}\label{p:1}  The inner product $\langle \cdot,\cdot \rangle$ is the one
of $L^{2}(\beta)$,  where $\beta $ is the arcsine law on $[-2,2]$ (see \eqref{e:abo}).
\begin{enumerate}
\item For any $C^{2}$ function $\phi $ on $[-2,2]$, 
\begin{equation}\label{ep:10000}
\begin{split}
\mathcal{E}\mathcal{N}\phi &= \phi - \int \phi\, d\beta , \\
\mathcal{N}\mathcal{E}\phi &=\phi - \int \phi\, d\beta.
\end{split}
\end{equation}
\item One has $\mathcal{E}\phi_{0}=0$, while for $n\ge1$, $\mathcal{E}\phi_{n}=\frac{1}{n}\phi_{n}$ and   $\mathcal{N}\phi_{n}=n\phi_{n}$ for any $n\ge0$.  In other words, $\mathcal{N}$ is the counting number operator for the Chebyshev basis in $L^{2}(\beta)$. 
\item For any, $\phi,\psi$,  $C^{1}$ functions on $[-2,2]$,  
\begin{equation}\label{ep:9}
\langle \mathcal{N}\phi,\psi \rangle=2\iint
\frac{(\phi(x)-\phi(y))(\psi(x)-\psi(y))}{(x-y)^{2}} \, \omega(dx\,dy).
\end{equation}
In particular, $\langle \mathcal{N}\phi,\psi \rangle=\langle \phi,\mathcal{N}\psi \rangle$.
\item If we take $\mathcal{L}\phi=\mathcal{N}^{2}\phi$ for $C^{2}$ functions, then 
\begin{equation}\label{ep:8}
\langle \mathcal{L}\phi,\psi\rangle=2\int \phi' \psi'd\alpha.
\end{equation}
The operator $\mathcal{L}$ is actually the Jacobi operator
\[
(\mathcal{L}\phi)(x)=-(4-x^{2})\phi''(x)+x\phi'(x)
\]
with invariant measure the arcsine law $\beta$.  Moreover, $\mathcal{L}$ has a unique selfadjoint extension, still denoted by $\mathcal{L}$ and defined on 
\[
\mathcal{H}=\{\phi:[-2,2]\to\R,\phi\in L^{2}(\beta) \text{ and } \phi(x)=\int_{-2}^{x}\psi(y)dy,\,\text{ for }\beta - \text{ a.s. } x\in[-2,2]\text{ with }\psi\in L^{2}(\alpha) \}.
\]

\item For a $C^{3}$ potential $V$ on $[-2,2]$, the solution $\mu$ to \eqref{e:0} is  
\[
\mu_{V}(dx)=\left(A-\frac{1}{2}\mathcal{N}V(x)\right)\beta(dx).
\] 
\item If the minimizer of $E_{V}$ on $[-2,2]$ has full support, then 
\begin{equation}\label{e:rEV}
E_{V}=\int V\,d\beta- \frac{1}{4}\langle \mathcal{N}V,V \rangle=\int V\,d\beta-\frac{1}{2}\iint\left( 
\frac{V(x)-V(y)}{x-y}\right)^{2}\omega(dx\,dy).
\end{equation}
\end{enumerate}
\end{proposition}

\begin{proof}
\begin{enumerate}
\item  We need to settle the fact that if $\phi$ is $C^{2}$ on $[-2,2]$, then $\mathcal{E}\phi$ is again $C^{2}$ on $[-2,2]$.  This is needed to give consistency to the second line of \eqref{ep:10000}. To this end, we try to remove the singularity in $\mathcal{E}$, by invoking \eqref{ep:6} and \eqref{ep:2} for the measure $\mu(dy)=y\beta(dy)$ to justify the following 
\[
\begin{split}
\mathcal{E}\phi(x)&=-\int \log|x-y|(\phi(y)-\phi(x)-\phi'(x)(y-x))\,\beta(dy)-\phi'(x)\int y\log|x-y|\,\beta(dy)\\
&=  -\int \log|x-y|(\phi(y)-\phi(x)-\phi'(x)(y-x))\,\beta(dy) +x\phi'(x),\quad \text{ for all }x\in[-2,2].
\end{split}
\] 
It is obvious from this writing and Lebesgue's dominated convergence that $\mathcal{E}\phi$ is actually a continuous function on $[-2,2]$.  Taking the derivative with respect to $x$ it is straightforward to deduce (again using \eqref{ep:6} and \eqref{ep:2} and dominated convergence) that 
\[
\begin{split}
\left(\frac{d}{dx}\mathcal{E}\phi\right)(x)&=\phi''(x)\int (y-x)\log|x-y|\,\beta(dy)+\int\frac{\phi(y)-\phi(x)-\phi'(x)(y-x)}{y-x}\,\beta(dy)+\phi'(x)+x\phi''(x)\\
&=\int\frac{\phi(y)-\phi(x)-\phi'(x)(y-x)}{y-x}\,\beta(dy)+\phi'(x)
\end{split}
\]
Taking again the derivative with respect to $x$ reveals that
\[
\begin{split}
\left(\frac{d^{2}}{dx^{2}}\mathcal{E}\phi\right)(x)&=\int\frac{\phi(y)-\phi(x)-\phi'(x)(y-x)}{(y-x)^{2}}\,\beta(dy)
\end{split}
\]
which shows that $\mathcal{E}\phi$ is actually $C^{2}$ if $\phi$ is $C^{2}$.  The rest now follows from Definition~\ref{d:1}
and Theorem~\ref{t:1}.  
\item It is an easy consequence of Lemma~\ref{l:1} and  \eqref{ep:20}.   
\item This is infered from \eqref{e:39}.
\item Equivalently, 
\[
\langle \mathcal{N}\phi,\mathcal{N}\psi\rangle=2\int \phi'\psi'd\alpha.
\]
In turn, it is sufficient to do this for $\phi=\phi_{n}$, $\psi=\phi_{m}$.  Thus we need only show that using \eqref {ep:TU}
\[
\int U_{n}\left( \frac{x}{2}\right)U_{m}\left( \frac{x}{2}\right)\alpha(dx)=\delta_{mn}.
\]
which is just the orthogonality of the polynomials $U_{n}\left( \frac{x}{2}\right)$ with respect to $\alpha$.  
The formula for $\mathcal{L}$ is just an integration by parts. 

The selfadjoint extension can be easily demonstrated by the fact that $\mathcal{L}$ has the eingenvalues $\{n^{2}\}_{n\ge0}$ with eigenfunctions $\phi_{n}$.    Indeed, it is easy to see that there is an isometry $\mathcal{A}:L^{2}(\beta)\to \ell^{2}(\N)=\{(a_{n})_{n\ge0}:\sum_{n\ge0}|a_{n}|^{2}<\infty\}$, which sends $\phi = \sum_{n\ge 0}a_{n}\phi_{n}$ into $(a_{n})_{n\ge0}$.  This isometry sends the operator $\mathcal{L}$ defined on the linear span of $\phi_{n}$ into the multiplication operator $\mathcal{R}(a_{n})_{n\ge0}=(n^{2} a_{n})_{n\ge0}$ on the space of sequences with finitely many nonzero entries.  Since the operator $\mathcal{R}$ has a unique selfadjoint extension, the same is true for $\mathcal{L}$.  The domain of $\mathcal{R}$ is pushed back by the inverse of $\mathcal{A}$ into $\mathcal{H}$. 

\item Just a rewriting of \eqref{e:sol}.  

\item Since  
\[
E_{V}=\int Vd\mu_{V}-\iint \log|x-y|\mu_{V}(dx)\mu_{V}(dy),
\]
where $\mu_{V}=\left(1-\frac{1}{2}\mathcal{N}V\right)d\beta$, it follows that 
\[
\begin{split}
E_{V}&=\int Vd\beta -\frac{1}{2}\langle\mathcal{N}V,V \rangle +\Big \langle \mathcal{E}\Big(1-\frac{1}{2}\mathcal{N}V\Big),\Big(1-\frac{1}{2}\mathcal{N}V\Big) \Big\rangle \\ 
&=  \int Vd\beta -\frac{1}{2}\langle\mathcal{N}V,V \rangle - \frac{1}{2} \Big\langle 
     \Big(V -\int Vd\beta \Big), \Big(1-\frac{1}{2}\mathcal{N}V \Big) \Big \rangle \\
&= \int V\,d\beta -\frac{1}{4} \, \langle \mathcal{N}V,V\rangle 
\end{split}
\]
which combined with \eqref{ep:9} gives \eqref{e:rEV}.\qedhere
\end{enumerate}
\end{proof}

We collect here some integration by parts properties of the operator $\mathcal{N}$ which will be used later on.  

\begin{theorem}\label{tp:3}
If $\mathcal{N}$ is the operator defined in \eqref{e:EN}, then for any two $C^{2}$ functions $\phi,\psi:[-2,2]\to\R$ 
\begin{equation}\label{ep:300}
\begin{split}
\langle \mathcal{N}\phi,\psi' \rangle+\langle \mathcal{N}\psi,\phi'\rangle &= \Pi(\phi')\Pi(x\psi')+\Pi(x\phi')\Pi(\psi') \\
\langle \mathcal{N}\phi,x\psi'\rangle+\langle \mathcal{N}\psi,x\phi'\rangle &= \Pi(x\phi')\Pi(x\psi')+4\Pi(\phi')\Pi(\psi'). 
\end{split}
\end{equation}
Here we use the notation $\Pi(\phi)=\int \phi\, d\beta$ and the convention that $x^{k}\phi$ is a shortcut for the function $f(x)=x^{k}\phi(x)$.
 In addition, 
\begin{equation}\label{ep:301}
\begin{split}
2\langle \mathcal{N}(\phi'),\psi' \rangle+\langle \mathcal{N}\psi,\phi''\rangle+\langle \mathcal{N}\phi,\psi''\rangle &= \Pi(\phi'')\Pi(x\psi')+\Pi(x\phi'')\Pi(\psi')  \\ 
&\quad+ \Pi(\phi')\Pi(x\psi'')+\Pi(x\phi')\Pi(\psi'') , \\
2\langle \mathcal{N}(x\phi'),x\psi' \rangle+\langle \mathcal{N}\psi,x^{2}\phi''\rangle+\langle \mathcal{N}\phi,x^{2}\psi''\rangle &= \Pi(x\phi')\Pi(x\psi')+\Pi(\phi')\Pi(\psi')  \\ 
& \quad + \Pi(x^{2}\phi'')\Pi(x\psi')+4\Pi(x\phi'')\Pi(\psi')  \\ 
& \quad + \Pi(x \phi')\Pi(x^{2}\psi'')+4\Pi(\phi')\Pi(x\psi'') , \\ 
\langle \mathcal{N}\phi,x\psi'' \rangle+\langle \mathcal{N}\psi,x\phi''\rangle+\langle \mathcal{N}(\phi'),x\psi'\rangle+\langle \mathcal{N}(\psi'),x\phi'\rangle &= \Pi(x\phi'')\Pi(x\psi')+4\Pi(\phi'')\Pi(\psi') \\ 
& \quad + \Pi(x \phi')\Pi(x\psi'')+4\Pi(\phi')\Pi(\psi'')  . 
\end{split}
\end{equation} 
The relation between $\mathcal{N}$ and $\mathcal{U}$ is that, for any $C^{2}$ function $f$ on $[-2,2]$,
\begin{equation}\label{ep:302}
(\mathcal{N}f)(x)=-\sqrt{4-x^{2}}\frac{d}{dx} \, \left[ \sqrt{4-x^{2}}(\mathcal{U}f)(x) \right].
\end{equation} 
\end{theorem}

\begin{proof}
It is clear that it is enough to check \eqref{ep:300} for $\phi=\phi_{n}$, $\psi=\phi_{m}$, in which case the first part becomes 
\[
nm\langle \psi_{n-1},\phi_{m} \rangle+mn\langle \phi_{n},\psi_{m-1}\rangle 
     =\frac{mn}{2} \, \Pi(\psi_{n-1})\Pi(x\psi_{m-1})+\frac{mn}{2} \, \Pi(x\psi_{n-1})\Pi(\psi_{m-1}).  
\]
This easily follows from
\begin{equation}\label{ep:100}
\langle \phi_{n},\psi_{m} \rangle= \begin{cases}
0&\text{ if } n>m  \, \, \text{ or } \, \, n-m=1  \, \, \, (\text{mod }2 ) \\
1&\text{ if } n\le m  \, \, \text{ and  } \, \, n-m=0 \,\, \, (\text{mod } 2).
\end{cases}
\end{equation}

To quickly see this, take the generating functions (here $0<r,w<1$)
$g_{r}$ and  $h_{w}$ already introduced in the proof of Lemma \ref {l:2} in \eqref{eq:GH} and observe that 
\[
g_{r}(x)h_{w}(x)=\frac{(1-rx/2)}{rw(a-x)(b-x)}=\frac{A}{a-x}+\frac{B}{b-x}
\]
with
\[
a=\frac{1+r^{2}}{r},\quad b=\frac{1+w^{2}}{w},\quad A=-\frac{1-r^2}{2 \left(w+r^2 w-r \left(1+w^2\right)\right)},\quad B=\frac{r-2 w+r w^2}{2 w \left(r-w-r^2 w+r w^2\right)}
\]
combined with the derivative of \eqref{ep:6} and a little algebra gives
\[
\begin{split}
\sum_{n, m  =0}^\infty r^{n}w^{m}\langle \phi_{n},\psi_{m} \rangle
 &=\int g_{r}h_{w}d\beta=\frac{A}{\sqrt{a^{2}-4}}+\frac{B}{\sqrt{b^{2}-4}} 
 =\frac{rA}{1-r^{2}}+\frac{wB}{1-w^{2}} 
\\ &=\frac{1}{(1-rw)(1-w^{2})}=\sum_{n ,k = 0} ^\infty  r^{n}w^{n+2k}
\end{split}
\]
which yields \eqref{ep:100}.  

For the second line of \eqref{ep:300},  it is again sufficient to look at $\phi=\phi_{n}$ and $\psi=\phi_{m}$, in which case we need to check that 
\[
nm\langle \psi_{n-1},x\phi_{m} \rangle+mn\langle x\phi_{n},\psi_{m-1}\rangle 
   =\frac{mn}{2} \, \Pi(x\psi_{n-1})\Pi(x\psi_{m-1})+ 2mn\, \Pi(\psi_{n-1})\Pi(\psi_{m-1}).  
\]
This follows also from \eqref{ep:100}, by observing that $\phi_{1}(x)=x/2$ and $x\phi_{n}=\phi_{n+1}+\phi_{n-1}$. 

To get the rest of the proof, notice that if we set, 
\[
J(u,v):= 2\iint\frac{[\phi(ux+v)-\phi(uy+v)] [ \psi(ux+v)-\psi(uy+v)] }{(x-y)^{2}} \,  \omega(dx\,dy),
\] 
then a simple scaling argument together with \eqref{ep:9} and \eqref{ep:300} imply
\begin{equation}\label{e:10}
\begin{split}
\frac{\partial J}{\partial v}=& u\left[\left(\int \phi'(ux+v)\beta(dx) \right)\left(\int x\psi'(ux+v)\beta(dx)  \right)+\left( \int x\phi'(ux+v)\beta(dx)  \right)\left(\int \psi'(ux+v)\beta(dx) \right) \right] , \\
\frac{\partial J}{\partial u}=& u\left[\left(\int x\phi'(ux+v)\beta(dx) \right)\left(\int x\psi'(ux+v)\beta(dx)  \right)+4\left( \int \phi'(ux+v)\beta(dx)  \right)\left(\int \psi'(ux+v)\beta(dx) \right) \right].
\end{split}
\end{equation}
Now differentiation with respect to $v$ at $(u,v)=(1,0)$ of the first equation gives the first line of \eqref{ep:301}, while the other two lines follow by differentiating with respect to $u$ and $v$ of the second equation above and setting $u=1,v=0$.  

To deal with \eqref{ep:302}, it  suffices to do this for $f=\phi_{n}$ and in fact in order to check the identity for each $n$, we take the generating functions instead of the left and right hand side, thus we need only check the following
\[
(\mathcal{N}g_{r})(x)=-\sqrt{4-x^{2}} \, \frac{d}{dx}\left[\sqrt{4-x^{2}} \, (\mathcal{U}g_{r})(x)\right].
\]
Now, since $\mathcal{N}$ is the counting number generator for $\phi_{n}$, the left hand side is actually equal to $\partial_{r}g_{r}(x)$, while the right hand side, from \eqref{ep:303}, gives $\mathcal{U}g_{r}=rh_{r}(x)/2$, in which case both sides give the same answer, 
namely $-\frac{4r^{2}-rx(1+r^{2})}{2(1+rx+r^{2})^{2}}$.
This completes the proof of Theorem \ref {tp:3}. \qedhere

\end{proof}

\section{Poincar\'e Inequality, General Properties}\label{s:3}

This section introduces the natural candidate for the free Poincar\'e inequality which is investigated
 throughout this note.

\begin{definition}
A probability
measure $\mu$ on  $[-2c+b,2c+b]$ is said to satisfy a free Poincar\'e inequality with constant $\rho>0$,
denoted $P (\rho)$, if 
\begin{equation}\label{e1:1}
 2\rho c^{2} \iint\left( \frac{f(x)-f(y)}{x-y}  \right)^{2}\omega_{b,c}(dx\,dy)\le\int (f')^{2}d\mu
\end{equation}
holds for any smooth $f$ on $[-2c+b,2c+b]$.
\end{definition}

 It should be observed that the left hand side in \eqref {e1:1} only
 depends on the measure $\mu $ through its support. Actually, the first assertion of
 Proposition \ref {p:2} below shows that $\mu $ has support $[-2c+b,2c+b]$.

The next statement collects some of the properties of this free Poincar\'e inequality.

\begin{proposition}\label{p:2}  Assume $\mu$ satisfies $P(\rho)$ on $[-2c+b,2c+b]$.   The following are true.  
\begin{enumerate}

\item $\mu$ has support $[-2c+b,2c+b]$. Moreover, if $d\mu=w\,d\alpha$, with $w\in C^{2}([-2c+b,2c+b])$, then $ w(x)>0$ for all $x\in[-2c+b,2c+b]$.  
\item The constant $\rho$ in \eqref{e1:1} satisfies 
\[ 
\rho\le \frac{1}{2c^{2}}
\] 
with equality if and only if $\mu=\alpha_{b,c}$. 

\item For  any $C^{1}$ function $f:[-2c+b,2c+b]\to\R$, 
\begin{equation}\label{e1:3}
\frac{1}{2} \, \Var_{\beta_{b,c}}(f)\le  \iint\left( \frac{f(x)-f(y)}{x-y}  \right)^{2}\omega_{b,c}(dx\,dy).
\end{equation}

In fact, this inequality is equivalent to  $P(1/2)$ for the semicircular
$\alpha_{b,c}$ with equality in \eqref{e1:3} or \eqref{e1:1} only for linear functions $f$.  

\item If $d\mu=w\,d\alpha_{b,c}$, with $w\ge \rho$ on $[-2c+b,2c+b]$, then $\mu$ satisfies $P(\frac{\rho}{2c^{2}})$.  
\end{enumerate}
\end{proposition}

\begin{remark} 
\eqref{e1:3} is actually a classical Poincar\'e inequality (spectral gap) for the operator $\mathcal{N}$ on $L^{2}(\beta)$ and it is equivalent (by item 3)  to the free Poincar\'e for the semicircular.  
\end{remark}

\begin{proof}
\begin{enumerate} 
\item It is pretty obvious that if $J$ is an interval with the property that $\mu(J)=0$, then choosing a function $f$ such that $f$ is constant outside the interval $J$ and is equal to $x$ on some smaller subinterval $K\subset J$ leads to a contradiction. 

One can not conclude that there is a density of $\mu$ with respect to the Lebesgue measure or for that matter with respect to the semicircular.  Indeed, for instance if we take $\mu=\frac{1}{2}\alpha_{b,c}+\frac{1}{2}\gamma$, with $\gamma$ a singular measure with respect to $\alpha_{b,c}$, then $\mu$ still satisfies a free Poincare and it is not absolutely continuous with respect to $\alpha_{b,c}$.  

Thus assume that $\mu=w\,\alpha_{b,c}$ with $w$ a continuous function.   We assume that $b=0$,  $c=1$.    In order to show that $w(a)>0$, for any $a\in (-2,2)$, we assume on the contrary that
$w(a)=0$ for some $a\in(-2,2)$.  Since $w(x)\ge0$ it follows that $a$ is a minimum point and thus, from the smoothness of $w$, $w(x)=w''(a)(x-a)^{2}+O((x-a)^{2})$.  

Now we choose an approximation of the identity constructed as follows.  First consider
\[
\phi(x)=\begin{cases}
e^{-\frac{1}{1-x^{2}}}& x\in[-1,1]\\
0 & \text{otherwise}.
\end{cases}
\]
Apply then the free Poincar\'e inequality to the function $f(x)=\phi((x-a)/\delta)$ to obtain that 
\[
\begin{split}
\rho\int_{\substack{|x-a|<\delta\\ |y-a|<\delta}} \left(\frac{\phi((x-a)/\delta)-\phi((y-a)/\delta}{x-y} \right)^{2}\omega(dx \, dy)&\le \rho\int \left(\frac{\phi((x-a)/\delta)-\phi((y-a)/\delta)}{x-y} \right)^{2}
  \omega(dx \,dy) \\ &\le \frac{1}{\delta^{2}}\int (\phi'((x-a)/\delta))^{2}w(x)\alpha(dx).
\end{split}
\]
Now, changing the variable $x=a+\delta x'$ and $y=a+\delta y'$, for small enough $\delta$, results with
\[
C_{a}\iint_{[-1,1]^{2}}  \left(\frac{\phi(x)-\phi(y)}{x-y} \right)^{2}dxdy\le\frac{1}{\delta}\int_{-1}^{1}(\phi'(x))^{2}w(a+\delta x)dx \le O(\delta)\int_{-1}^{1}(\phi'(x))^{2}dx 
\]
where $C_{a}>0$ is a constant depending on $a$ and $\rho$.  Hence we get  a contradiction as we let $\delta\to0$.  Therefore on $(-2,2)$, the density $w$ must be positive.  

Now we deal with the behavior at the edge.  Assume $w(-2)=0$.  The vanishing of $w$ near $-2$ is no longer of order $2$, but of order $1$. Thus $w(x)=(x+2)w'(-2)+o((x+2)^{2})$.   Take again $f(x)=\phi((x+2)/\delta)$ and apply the free Poincar\'e, to obtain 
\[
\begin{split}
\rho\int_{\substack{-2<x<-2+\delta \\-2<y<-2+\delta}} \left(\frac{\phi((x+2)/\delta)-\phi((y+2)/\delta}{x-y} \right)^{2}\omega(dx \, dy)&\le \rho\int \left(\frac{\phi((x+2)/\delta)-\phi((y+2)/\delta)}{x-y} \right)^{2}\omega(dx \, dy) \\ &\le \frac{1}{\delta^{2}}\int (\phi'((x+2)/\delta))^{2}w(x)\alpha(dx).
\end{split}
\]
Make the change of variables $x=-2+\delta x'$, $y=-2+\delta y'$ and deduce that for a constant $C>0$,
\[
C\iint_{[0,1]^{2}} \left( \frac{\phi(x)-\phi(y)}{x-y} \right)^{2}dxdy
   \le O(\sqrt{\delta})\int_{0}^{1} (\phi'(x))^{2}dx
\]
where we used that 
\[
\frac{4-(-2+\delta x')(-2+\delta y')}{\sqrt{(4-(-2+\delta x')^{2})(4-(-2+\delta y')^{2})}}\ge C>0
\]
uniformly for $x',y'\in[0,1]$ and small $\delta$.  Consequently, letting $\delta\to0$, we arrive to  a contradiction.   

\item Taking in \eqref{e1:1} $f(x)=x$, it is immediate that $2\rho c^{2}\le 1$.  Now, conversely, assume that $\rho=1/(2c^{2})$.   We may assume
that $b=0$, $c=1$, $\rho=1/2$ and that the measure $\mu$ is supported on $[-2,2]$.  
Take $f(x)=rx+\phi(x)$.  Then the Poincar\'e implies that for any $r\in\R$,
\[
\iint\left( \frac{\phi(x)-\phi(y)}{x-y}  \right)^{2}\omega(dx\,dy)+r\iint \frac{\phi(x)-\phi(y)}{x-y} \omega(dx\,dy)+r^{2}\le \int (\phi')^{2} d\mu+2r\int \phi'(x)\mu(dx)+r^{2}.
\]
Consequently, 
\[
\iint \frac{\phi(x)-\phi(y)}{x-y} \, \omega(dx\,dy)=2\int \phi'(x)\mu(dx).
\]
In particular we can rewrite this in terms of the operators, $\mathcal{N}$ and $\mathcal{L}$ and the notations from Proposition~\ref{p:1}
\[
\langle \mathcal{N}\phi,\phi_{1} \rangle=2\int \phi'd\mu
\]
On the other hand, since $N\phi_{1}=\phi_{1}$ combined with \eqref{ep:8} give
\[
\int \phi'd\mu=\int\phi' d\alpha.  
\]
which shows that $\mu=\alpha$.  

\item It suffices to do this for the case of $b=0$, $c=1$.   From Lemma~\ref{l:2} we know that $\mathcal{U}\phi_{n}=\frac{1}{2}\psi_{n}$ and then  writing  $f=\sum_{n =1}^\infty a_{n}\phi_{n}$ and keeping in mind \eqref{ep:9},  \eqref{e1:3} becomes equivalent to 
\[
\sum_{n =1} ^\infty a_{n}^{2}\le \sum_{n =1} ^\infty n a_{n}^{2}
\]
which is obviously true.  Written in terms of the operator $\mathcal{N}$, \eqref{e1:3} is equivalent to 
\[
\Var_{\beta}(f)\le \langle\mathcal{N}f,f \rangle
\]
for all $f\in C^{2}([-2,2])$.  Or this is just the spectral gap of $\mathcal{N}$.   Equality is attained in \eqref{e1:3} only for $f$ linear.  The free Poincar\'e's is actually equivalent to the statement $\mathcal{N}\le \mathcal{N}^{2}$.  As $\mathcal{N}$ is a non-negative operator, this is in fact equivalent to \eqref{e1:3}.   

\item Follows from  $P(1/2)$ for $\alpha$.\qedhere
\end{enumerate}
\end{proof}

\begin{remark}
Poincar\'e's inequality and the $C^{2}$ condition on the density $w$ imply
that $w$ must be positive.  Also, if $w$ is positive everywhere and continuous
then $P(\rho)$ holds for some $\rho$.  
It is interesting to see what happens if the $C^{2}$ condition on $w$ is dropped.  Is it still true that there is a Poincar\'e inequality satisfied for some $\rho>0$?   And if so, under what are the regularity conditions on $w$?   
\end{remark}

\begin{remark} 
A natural question in this context is about the extension of the Poincar\'e to the case where the measure $\mu$ has more then one interval support.  These arise naturally as equilibrium measures $\mu_{V}$ for potentials $V$ with several wells.  Indeed, it was shown in \cite{Deift2} that if $V$ is analytic near the support of $\mu_{V}$, then the support of $\mu_{V}$ must be a finite union of intervals.   
If a probability measure  $\mu$ is supported on a finite number of intervals,
say $I_{1}\cup I_{2}\dots \cup I_{k}$, and satisfies 
\[
c\int \left( \frac{f(x)-f(y)}{x-y}\right)^{2}\gamma(dx\,dy)\le \int (f')^{2}d\mu
\]
for all smooth functions on $\R$, then it can be shown, that each restriction of $\mu_{m}$ to each connected component $I_{m}$ satisfies an inequality of the form 
\[
c\int \left( \frac{f(x)-f(y)}{x-y}\right)^{2}\gamma_{m}(dx\,dy)\le \int (f')^{2}d\mu_{m}
\]
with $\gamma_{m}$ supported on $I_{m}\times I_{m}$.  

\end{remark}

\section{Equivalent Forms of Poincar\'e's Inequality}\label{s:4}

In this section we discuss the various equivalent forms of the free Poincar\'e
inequality \eqref{e1:1}.    Before we do this, let us introduce some operators.  

For a given measure $\mu=w\,d\alpha$, with $w\in C^{1}([-2,2])$  let $\mathcal{L}_{w}$ be the operator acting on $L^{2}(\beta)$ with the Dirichlet form given by $2\int (f')^{2}\mu$.  Then an integration by parts gives 
\[
\begin{split}
\langle\mathcal{L}_{w}\phi,\psi\rangle&=2\int \phi'\psi' wd\alpha = -\frac{1}{\pi}\int_{-2}^{2} \psi \frac{d}{dx} \big (\phi' w\sqrt{4-x^{2}} \big )dx 
= \int  \left(-(4-x^{2})w\phi''+(xw-(4-x^{2})w')\phi'\right)\psi \,\beta(dx),
\end{split}
\]
from which 
\[
\mathcal{L}_{w}\phi=-(4-x^{2})w\phi''+(xw-(4-x^{2})w')\phi'.
\]
Notice that for the case $w=1$, the operator $\mathcal{L}_{w}$ becomes $\mathcal{L}$ given in part (4) of Proposition~\ref{p:1}. 

Here is a statement which will be used in the sequel.

\begin{proposition}\label{p:4} If $w>0$ on $[-2,2]$ and in $C^{2}([-2,2])$, the operator $\mathcal{L}_{w}$ extends to a selfadjoint operator on $L^{2}(\beta)$ with domain $\mathcal{H}$, defined in part (4) of Proposition~\ref{p:1}.  
\end{proposition}

\begin{proof}  It is clear that $\mathcal{L}_{w}$ sends the constant functions to $0$ and thus we restrict our attention to the restriction of $\mathcal{L}_{w}$ on the orthogonal to constants in $L^{2}(\beta)_{0}$, which the set of functions in $L^{2}(\beta)$ of mean $0$.  

There is another way of representing this operator as 
 \[
 \mathcal{L}_{w}f=\mathcal{L}\mathcal{A}_{w}f
 \] 
 with 
 \[
 (\mathcal{A}_{w}f)(x)=\int_{-2}^{x}f'(y)w(y)dy-\int_{-2}^{2} \int_{-2}^{x}f'(y)w(y)dy\,\beta(dx).
 \] 
 for any $C^{2}$ function $f$ on $[-2,2]$.   It is not hard to check that the operator $\mathcal{A}_{w}$ can be extended to a bounded operator on $L^{2}_{0}(\beta)$ due to the fact that $w$ is $C^{1}$. In addition, it maps $\mathcal{H}_{0}=\mathcal{H}\cap L^{2}_{0}(\beta)$  into itself and has the inverse on $L^{2}_{0}(\beta)$ given by $\mathcal{A}_{1/w}$.   In particular, we can use this to extend the operator $\mathcal{L}_{w}$ to $\mathcal{H}_{0}$.     
 
 The claim is now that this operator is actually selfadjoint.  Indeed, if $\psi\in L^{2}_{0}(\beta)$, which in the domain of $\mathcal{L}_{w}^{*}$, then by definition, $\phi\to\langle \mathcal{L}\mathcal{A}_{w}\phi,\psi \rangle$ extends to a  bounded functional from $L^{2}_{0}(\beta)$ into $\R$.  Thus, there is a constant $C>0$ such that $\langle \mathcal{L}\mathcal{A}_{w}\phi,\psi \rangle\le C\| \phi\|$, say for any $C^{2}$ function $\phi\in C^{2}([-2,2])\cap L^{2}_{0}(\beta)$ and then replacing $\phi$ by $\mathcal{A}_{1/w}\phi$ and the fact that this is bounded we obtain that $\langle \mathcal{L}\phi,\psi \rangle\le C\| \mathcal{A}_{1/w}\|\| \phi \|$ for any $C^{1}$ function $\phi$ on $[-2,2]$ in $L^{2}_{0}(\beta)$.  Hence $\psi$ is in the domain of $\mathcal{L}^{*}$, which is $\mathcal{H}_{0}$ by the fourth item of Proposition~\ref{p:1}.  In particular this means that the domain of $\mathcal{L}_{w}^{*}$ is $\mathcal{H}_{0}$.  
 
 On the other hand, since $\mathcal{L}_{w}$ on $\mathcal{H}_{0}$ is the closure of the same operator restricted to $C^{2}([-2,2])\cap \mathcal{H}_{0}$, it follows that $\mathcal{L}_{w}$ and $\mathcal{L}_{w}^{*}$ have the same domain of definition, namely $\mathcal{H}_{0}$ and thus $\mathcal{L}_{w}$ on $\mathcal{H}_{0}$ is selfadjoint.  
 
 \end{proof}

Recall the operator $\mathcal{U}$, which is defined in Lemma~\ref{l:1} and for which
$ \mathcal{U}\phi_{n}=\frac{1}{2}\psi_{n-1}$.
It is natural to look at this operator between $L^{2}(\beta)$ and $L^{2}(\alpha)$.  In this form,
\[
\| \mathcal{U}f \|^{2}_{\alpha}=\frac{1}{2}\Var_{\beta}(f).
\]
Now we define the inverse operator of $\mathcal{U}$ by 
\begin{equation}\label{eq:calV}
\mathcal{V}\psi_{n}=2\phi_{n+1}\quad \text{ for } \, \, n\ge0.  
\end{equation}
It is clear in this case that 
\[
\|\mathcal{V}f\|_{\beta}^{2}=2\|f \|^{2}_{\alpha}  
\]
or equivalently, 
\[ 
\langle \mathcal{V}\phi,\mathcal{V}\psi \rangle_{\beta}=2\langle\phi,\psi \rangle_{\alpha}.
\]
Also we have 
\[
\mathcal{UV}=I \quad \text{ and } \quad \mathcal{VU}=I-\Pi
\]
where $\Pi$ is as above the projection on constant functions in $L^{2}(\beta)$.  

On smooth functions $\phi$, the operator $\mathcal{V}$ has an explicit form as
\[
(\mathcal{V}\phi)(x) = \Pi(y\phi)+x\Pi(\phi)-(4-x^{2})(\mathcal{U}\phi)(x). 
\]
It is easy to see that one has to check this on the generating function of $\psi_{n}$, which is $h_{r}(x)=\frac{1}{1-rx+r^{2}}$, $0<r<1$.  For such a particular function, (cf. \eqref{ep:30})
\[
\begin{split}
x\Pi(h_{r})&=x\int h_{r}d\beta=\frac{x}{1-r^{2}} \, , \\
\Pi(yh_{r}) &=\int yh_{r}(y)\,\beta(dy)=\frac{2r}{1-r^{2}} \, ,\\ 
(4-x^{2})(\mathcal{U}h_{r})(x)&= \frac{2r}{1-r^{2}}+\frac{x}{1-r^{2}}-\frac{2(g_{r}(x)-1)}{r}
\end{split}
\]
which gives the formula.  The point of the formula is that for a $C^{2}$ function $f$ on $[-2,2]$, $\mathcal{V}f$ is at least $C^{1}$.  

Now take 
\[
\mathcal{M}=\mathcal{UNV}-I,
\]
where $I$ is the identity operator.   It is very easy to check that $\mathcal{M}$ is the counting number operator for the $\{\psi_{n}\}_{n\ge0}$ basis of $L^{2}(\alpha)$ for the semicircle law. 
Indeed, on the basis $\psi_{n}$, both sides give $n\psi_{n}$.    With this definition, it is easy to check  that
\begin{equation}\label{e1:34}
\mathcal{N}\mathcal{V}=\mathcal{V}(\mathcal{M}+I).
\end{equation}
We also have 
\begin{equation}\label{eq:mathM}
\langle \mathcal{M}g,g\rangle_{\alpha}=\iint \left( \frac{g(x)-g(y)}{x-y}\right)^{2}\alpha(dx)\alpha(dy),
\end{equation}
which stems from the fact that
\[
\frac{\psi_{n}(x)-\psi_{n}(y)}{x-y}=\sum_{k=0}^{n-1}\psi_{k}(x)\psi_{n-k-1}(y)
\]
(a consequence of the generating function for $\psi_{n}$'s) used in conjunction with the  orthogonality of $\{\psi_{n}\}_{n\ge0}$ with respect to the measure $\alpha$.

The next theorem describes equivalent description of the free Poincar\'e inequality
$P(\rho)$ which follow from the preceding operator-theoretic tools.
Recall that $ \mathcal{U}_{b,c}$ appearing below is the one defined in Lemma \ref {l:2}.  

\begin{theorem}\label{t:eq}  Assume that $\mu=w\,\alpha_{b,c}$ with $w\in C^{2}([-2c+b,2c+b])$ and $\rho>0$.  Then the following are equivalent
\begin{enumerate}
\item  $P(\rho)$ for $\mu$ (\eqref{e1:1}).
\item For any $f\in C^{2}[-2c+b,2c+b]$
\begin{equation}\label{e1:4}
2\rho \int \frac{(\mathcal{U}_{b,c} f)^{2}}{w}\,d\alpha_{b,c}\le \iint\left( \frac{f(x)-f(y)}{x-y}  \right)^{2}\omega_{b,c}(dx\,dy).
\end{equation}
We call this alternative version $P_{2}(\rho)$.
\item For any $f\in C^{2}[-2c+b,2c+b]$, $\int \frac{(\mathcal{U}_{b,c}f)^{2}}{w}\,d\alpha_{b,c}<\infty$ and
\begin{equation}\label{e1:4b}
c^{2}\iint\left( \frac{f(x)-f(y)}{x-y}  \right)^{2}\omega_{b,c}(dx\,dy)\le 2\sqrt{\int (f')^{2}\,d\mu}\sqrt{\int \frac{(\mathcal{U}_{b,c}f)^{2}}{w}\,d\alpha_{b,c}} 
    - 2\rho \int \frac{(\mathcal{U}_{b,c}f)^{2}}{w}\,d\alpha_{b,c}.
\end{equation}
 We call this inequality $P_{3}(\rho)$. 
 
\item For any $g\in C^{1}([-2c+b,2c+b])$, 
\begin{equation}\label{e1:4b2}
2\rho \int \frac{g^{2}}{w}d\alpha_{b,c} \le c^{2}\iint \left( \frac{g(x)-g(y)}{x-y}\right)^{2}\alpha_{b,c}(dx)\alpha_{b,c}(dy)+\int g^{2}d\alpha_{b,c},
\end{equation}
which is referred to as  $P_{4}(\rho)$.
\end{enumerate}

\end{theorem}

\begin{proof}  We prove that $(1)$ implies $(2)$ implies $(3)$ implies $(1)$  and that  $(2)$ is equivalent to $(4)$.  In addition, even though it is not needed, we will also prove that $(2)$ implies $(1)$ with the duality argument which shows $(1)$ implies $(2)$.  This last implication makes more transparent the duality behind $P(\rho)$ and $P_{2}(\rho)$.  By scaling it can be assumed that $b=0,c=1$.

$(1)\implies (2)$  From Proposition~\ref{p:2} we learn that $w>0$ on $[-2,2]$.   Write $P(\rho)$ in
the equivalent form 
\[
2\rho\mathcal{N}\le \mathcal{L}_{w}
\]
as (unbounded) selfadjoint operators on $L^{2}(\beta)$.    Since $w>0$, then we can find two positive constants $c_{1},c_{2}>0$, such that 
\[
c_{1}\mathcal{L}\le \mathcal{L}_{w}\le c_{2}\mathcal{L}.
\]
Notice that the kernel of both $\mathcal{N}$ and $\mathcal{L}_{w}$ is the space of constant functions and therefore the restrictions of $\mathcal{N,L}_{w}$ to $L^{2}_{0}(\beta)$, the orthogonal to constant functions, are invertible.  We will assume for the rest of this implication that the operators $\mathcal{N,L}_{w}$ are taken on $L^{2}_{0}(\beta)$.  As the inverse of $\mathcal{N}$ is $\mathcal{E}$ and this is bounded, it follows that $\mathcal{L}_{w}^{-1}$ is also bounded.   

After these preliminaries, we use some sort of duality.  More precisely, the main idea is that for each fixed $f\in L^{2}_{0}(\beta)\cap C^{2}([-2,2])$, 
\begin{equation}\label{eq:dual}
\begin{split}
\sup_{g\in L^{2}_{0}(\beta)\cap C^{2}([-2,2])}\left\{ \langle \mathcal{N}f,g \rangle -\rho \, \langle\mathcal{N}g,g \rangle \right\}&=\frac{1}{4\rho} \, \langle \mathcal{N}f,f \rangle\\
\sup_{g\in L^{2}_{0}(\beta)\cap C^{2}([-2,2])}\left\{ \langle \mathcal{N}f,g \rangle -\frac{1}{2} \, \langle\mathcal{L}_{w}g,g \rangle \right\}&=\frac{1}{2} \, \langle \mathcal{N}\mathcal{L}_{w}^{-1}\mathcal{N}f,f \rangle.
\end{split}
\end{equation}
Indeed, the first equality is a consequence of $\langle \mathcal{N}(f-2\rho g),f-2\rho g\rangle \ge0$ for each $f,g\in L^{2}_{0}(\beta)\cap C^{2}([-2,2])$, while the second follows from $\langle \mathcal{L}_{w}^{-1}\mathcal{N}(f-\mathcal{EL}_{w}g),\mathcal{N}(f-\mathcal{EL}_{w}g) \rangle\ge0$, with equality for $g=\mathcal{L}_{w}^{-1}\mathcal{N}f$.  This last equality may not be attained for  $g\in L^{2}_{0}(\beta)\cap C^{2}([-2,2])$, but $\mathcal{L}_{w}^{-1}\mathcal{N}f$ can be approximated by such functions. 

Poincar\'e's inequality $P(\rho)$ implies in this case that 
\[
2\rho\langle \mathcal{N}\mathcal{L}_{w}^{-1}\mathcal{N}f,f \rangle\le \langle \mathcal{N}f,f \rangle.
\]
A simpler argument of this inequality was suggested by the reviewer of this paper and is based on the fact that from $2\rho \mathcal{N}\le\mathcal{L}_{w}$ on $L^{2}_{0}(\beta)$, we get first $2\rho \mathcal{L}_{w}^{-1}\le \mathcal{N}^{-1}$ and then $2\rho\mathcal{N}\mathcal{L}_{w}^{-1}\mathcal{N}\le \mathcal{NN}^{-1}\mathcal{N}=\mathcal{N}$.  

To get to \eqref{e1:4}, it suffices  to observe that for $f\in C^{2}([-2,2])\cap L^{2}_{0}(\beta)$
\[
\int \frac{(\mathcal{U}f)^{2}}{w}d\alpha
    = \Big \langle \mathcal{U}f,\frac{1}{w} \, \mathcal{U}f\Big \rangle_{\alpha}
      = \frac{1}{2}\left\langle \mathcal{V} \, \frac{1}{w} \, \mathcal{U}f,f\right\rangle_{\beta}.
\]
It remains now to show that $\mathcal{N}\mathcal{L}_{w}^{-1}\mathcal{N}=\mathcal{V}\frac{1}{w}\mathcal{U}$ on $ C^{2}([-2,2])\cap L^{2}_{0}(\beta)$.  Passing to the inverses, this follows from the following result which is remarkable enough to be called a Lemma.   

\begin{lemma} For any $w\in C^{2}([-2,2])$, 
\begin{equation}\label{eq:equal}
\mathcal{V}w\mathcal{U}=\mathcal{E}\mathcal{L}_{w}\mathcal{E} \, \,  \text{ on } \, \, C^{2}([-2,2])\cap L^{2}_{0}(\beta).
\end{equation}
\end{lemma}

\begin{proof}
It suffices to do this for $w=\phi_{n}$.  Therefore we need to check that 
\[
\mathcal{V}\phi_{n}\mathcal{U}\phi_{m}=\mathcal{E}\mathcal{L}_{\phi_{n}}\mathcal{E}\phi_{m}
\]
for all $m\ge1$ and $n\ge0$.   It is clear that 
\[
\mathcal{L}_{\phi_{n}}\phi=\phi_{n}\mathcal{L}\phi-\frac{n(4-x^{2})}{2} \, \psi_{n-1}\phi'.
\]
Now we can continue with
\[
\mathcal{V}\phi_{n}\mathcal{U}\phi_{m}=\mathcal{E}\mathcal{L}_{\phi_{n}}\mathcal{E}\phi_{m},
\]
or
\[
\frac{1}{2} \, \mathcal{V}\phi_{n}\psi_{m-1}
    =\frac{1}{m} \, \mathcal{E}\mathcal{L}_{\phi_{n}}\phi_{m}
       =\frac{1}{m} \, \mathcal{E}\phi_{n}\mathcal{L}\phi_{m}-n\mathcal{E}(4-x^{2})\psi_{n-1}\psi_{m-1}
       =m\mathcal{E}\phi_{n}\phi_{m} - n\mathcal{E} \, \frac{(4-x^{2})}{4}\psi_{n-1}\psi_{m-1}.
\]
From \eqref{eq:phipsi}, this is equivalent to 
\[
\frac{1}{4}\mathcal{V}({\mathrm{sign}}(m-n)\psi_{|m-n|-1}+\psi_{n+m-1})
  =\frac{m}{2} \, \mathcal{E}(\phi_{|n-m|}+\phi_{n+m})-\frac{n}{2} \, \mathcal{E}(\phi_{|n-m|}-\phi_{n+m})
\]
which becomes obvious based on \eqref{eq:calV} and part 2 of Proposition \eqref{p:1}.  Just as a clarification, $\mathrm{sign}(x)$ is $-1$ for $x<0$, $0$ for $x=0$ and $1$ for $x>0$.\qedhere
\end{proof}

$(2)\implies(1)$ We present two proofs for this implication.  The first one is a duality argument like the one used in the previous implication and the second one is based on \eqref{ep:302} and integration by parts.   

Before we launch into the proofs, let us point out that if $\int \frac{(\mathcal{U}f)^{2}}{w}\,d\alpha$ is finite for any $C^{2}$ function $f$, then $w(a)>0$ for $a\in(-2,2)$ and either $w(-2)>0$ or $w(-2)=0$ and  $w'(-2)>0$.  Similarly, $w(2)>0$ or $w(2)=0$ and $w'(2)>0$.  Indeed, if $w(a)=0$ for some interior point $a\in(-2,2)$, then, since $w\ge0$ and in $C^{2}$, it means that $w(x)=O((x-a)^{2})$ near $a$.  On the other hand we can find an $n$ such that $U\phi_{n}=\psi_{n-1}/2$ is nonzero in a neighborhood of $a$.  To see this, recall that $\psi_{n}(2\cos \theta)=\sin((n+1)\theta)/\sin(\theta)$, thus for any $\theta\in(0,\pi)$, there is $n$ such that $\sin((n+1)\theta)\ne 0$.  Combining these two facts, it easily leads to a contradiction of \eqref{e1:4}.  If $w'(-2)=0$, then $w$ vanishes quadratically near $-2$ and for instance, picking $f(x)=x$, leads to a contradiction.  

The first proof is based on the following duality similar to \eqref{eq:dual}.  For any $f\in L^{2}_{0}(\beta)\cap C^{2}([-2,2])$,
\begin{equation}\label{eq:dual2}
\begin{split}
\sup_{g\in L^{2}_{0}(\beta)\cap C^{2}([-2,2])}\Big\{ \langle \mathcal{N}f,g \rangle -\frac{1}{4\rho} \, \langle\mathcal{N}g,g \rangle \Big\}&=\rho \, \langle \mathcal{N}f,f \rangle\\
\sup_{g\in L^{2}_{0}(\beta)\cap C^{2}([-2,2])}\Big\{ \langle \mathcal{N}f,g \rangle -\frac{1}{2} \, \langle\mathcal{N}\mathcal{L}_{w}^{-1}\mathcal{N}g,g \rangle \Big\}&=\frac{1}{2} \, \langle \mathcal{L}_{w}f,f \rangle \\ 
\sup_{g\in L^{2}_{0}(\beta)\cap C^{2}([-2,2])}\Big\{ \langle \mathcal{N}f,g \rangle -\frac{1}{2} \, \langle \mathcal{V}\frac{1}{w}\mathcal{U}g,g \rangle \Big\}&=\frac{1}{2} \, \langle \mathcal{L}_{w}f,f \rangle .
\end{split}
\end{equation}
The first two equalities can be justified as in the previous proof, the last line being just the consequence of the above Lemma.  As the second form in Theorem~\ref{t:eq} is written as $2\rho\langle \mathcal{V}\frac{1}{w}\mathcal{U}g,g\rangle\le \langle\mathcal{N}g,g \rangle$, $P(\rho)$ is immediate.

The second proof is based on the idea that from \eqref{ep:9} and \eqref{ep:302}, a simple integration by parts yields
\[
2 \iint\left( \frac{f(x)-f(y)}{x-y}  \right)^{2}\omega(dx\,dy)=\langle\mathcal{N}f,f \rangle=2\int f' \mathcal{U}f \,d\alpha.
\]
Therefore, \eqref{e1:4} implies $P(\rho)$ from the following sequence 
\begin{equation}\label{e4:100}
\iint\left( \frac{f(x)-f(y)}{x-y}  \right)^{2}\omega(dx\,dy) \le  2\int f' \mathcal{U}f\,d\alpha-2\rho\int \frac{(\mathcal{U}f)^{2}}{w}d\alpha\le \frac{1}{2\rho}\int (f')^{2}wd\alpha=\frac{1}{2\rho}\int (f')^{2}d\mu,
\end{equation}
where the second inequality is justified by $2ab\le a^{2}+b^{2}$ with $a=f'/\sqrt{\rho}$ and $b=2\sqrt{\rho}\,\mathcal{U}f/\sqrt{w}$.   Notice here that we need to know that $w$ does not vanish on $(-2,2)$ and we have the suitable integrability of $1/w$ at $\pm2$ to ensure the integrals are well defined. 

$(2)\implies (3)$  The first inequality of \eqref{e4:100}, gives, after an application of the integral Cauchy-Schwarz inequality, 
\[
\iint\left( \frac{f(x)-f(y)}{x-y}  \right)^{2}\omega(dx\,dy) \le  2\int f' \mathcal{U}f\,d\alpha-2\rho\int \frac{(\mathcal{U}f)^{2}}{w}d\alpha\le 2\sqrt{\int (f')^{2}d\mu}\sqrt{ \int \frac{(\mathcal{U}f)^{2}}{w}\,d\alpha}-2\rho\int \frac{(\mathcal{U}f)^{2}}{w}d\alpha. 
\]

$(3)\implies (1)$ It is just an application of the Cauchy-Schwarz inequality.  More precisely, in $2\sqrt{ab}\le a+b$, $a,b\ge0$, take $a=\frac{1}{2\rho}\int (f')^{2}d\mu$ and $b=2\rho\int \frac{(\mathcal{U}f)^{2}}{w}d\alpha$.  We need to point out here that for all $C^{2}$ functions $f$, $\int \frac{(\mathcal{U}f)^{2}}{w}d\alpha<\infty$, hence, as it was shown in the implication $(2)\implies (1)$, $w$ must be positive inside $(-2,2)$ and is such that $1/w$ is $\alpha$-integrable.

$(2)\implies (4)$ Take now $g=\mathcal{U}f$, with $f=\mathcal{V}g$.
Therefore, if we replace $f$ in \eqref{e1:4} by $\mathcal{V}g$, then 
\[
4\rho\int \frac{g^{2}}{w}d\alpha \le \langle \mathcal{NV}g, \mathcal{V}g \rangle_{\beta}=\langle \mathcal{V}(\mathcal{M}+I)g,\mathcal{V}g \rangle_{\beta} 
=  2\langle (\mathcal{M}+I)g,g\rangle_{\alpha} 
\]
which is  exactly \eqref{e1:4b2}.  

$(4)\implies (2)$  Take $g=\mathcal{U}f$ in \eqref{e1:4b2} and from the last equation and $\mathcal{VU}=I-\Pi$, 
 \[
 2\langle\mathcal{NVU}f,\mathcal{VU}f\rangle =2\langle\mathcal{N}f,f\rangle =\langle (\mathcal{M}+I)g,g\rangle_{\alpha}
 \] 
 where we used that $\Pi$ is the projection onto the constant functions which is also the kernel of $\mathcal{N}$, thus $\mathcal{N}\Pi=\Pi\mathcal{N}=0$. \qedhere
\end{proof}

\begin{remark} 
It is interesting that  the equivalence of the first and second part of Theorem~\ref{t:eq} can be seen as some sort of duality.  

As we will see in Theorem~\ref{t:5}, the second form of Poincar\'e $P_2 (\rho)$ 
is naturally derived from the transportation inequality and this is the reason why we discuss this equivalent form.  At first we arrived from the transportation inequality to
\begin{equation}\label{e4:200}
\iint\left( \frac{f(x)-f(y)}{x-y}  \right)^{2}\omega(dx\,dy) \le  2\int f' \mathcal{U}f\,d\alpha-2\rho\int \frac{(\mathcal{U}f)^{2}}{w}d\alpha,
\end{equation}
which is a rewriting  of $P_{2}(\rho)$ from which a straightforward application of the Cauchy's inequality implies $P(\rho)$.   This makes one believe that the second form is actually stronger than $P(\rho)$ but the above  theorem says that they are equivalent.  

The third form is \eqref{e4:200} plus Cauchy-Schwarz.  This actually appears naturally from the HWI inequality  discussed in Section~\ref{s:fineq}.

The fourth form is closer in spirit to the classical form of Poincar\'e as a spectral gap, though a little different.  For example in the case of the semicircular on $[-2,2]$,  $w=1$ and this inequality becomes, 
\[
\| g\|_{2}^{2}\le \langle \mathcal{M}g, g \rangle_{\alpha}+\| g\|_{2}^{2}
\]
which is nothing but non-negativity of $\mathcal{M}$ on $L^{2}(\alpha)$.  This is to be put in contrast with Biane's version \eqref{eq:pb} of Poincar\'e's which is actually a measure of the spectral gap of $\mathcal{M}$.  
\end{remark}

\begin{remark}[The optimality of the constant $\rho$ in $P(\rho)$]

$P(\rho)$ becomes $2\rho \mathcal{N}\le \mathcal{L}_{w}$.  This inequality gives in particular that if $0=\lambda_{0}<\lambda_{1}\le \lambda_{2}\dots$ are the eigenvalues of $L_{w}$ ordered non-decreasingly, then $2\rho n\le \lambda_{n}$.  The optimal $\rho$ is the infimum of $\lambda_{n}/n$ over $n\ge1$.   On the other hand, if $\inf w>0$, then $\lambda_{n}$ grow at least quadratically and as such, there is a finite $n$, for which $\lambda_{n}=2\rho n$, $\lambda_{m}>2\rho m$ for $m=1,2,\dots, n-1$ and $\lambda_{m}\ge2\rho m$ for all $m\ge n+1$.   In some sense, the optimality constant is fitting the best linear growth for the spectrum of $\mathcal{L}_{w}$.   

From the point of view of $P_{2}(\rho)$, we are looking at the best constant of something which resembles a classical Poincar\'e inequality, as the left hand side of \eqref{e1:4} is some sort of variance.   However, unless $w$ is constant, the isometric property of $\mathcal{U}$ between $L^{2}(\beta)$ and $L^{2}(\alpha)$ is disturbed.

$P_{4}(\rho)$ is comparing  $\mathcal{M}+I$ with respect to the identity on a different  $L^{2}$.  
\end{remark}

\section{Perturbation of Logarithmic Potentials}\label{s:5}

In this section we provide some results related to logarithmic potentials which are the building blocks for the connection of transportation and Poincar\'e.  The goal is to study the result of
a perturbation of $V$ on $E_{V}$.
First recall the following result from \cite{GP} which gives an expression for $E_{V}$,
rewritten here  within the notations introduced so far.

\begin{theorem}\label{t:51}
Assume $V$ is a $C^{3}$ potential.  Then the equilibrium measure on $\R$ 
associated to $V$ has support the interval $[-2c+b,2c+b]$ if and only if 
$(c,b)$ is the unique absolute maximizer of  
\begin{equation}\label{e:31}
H(c,b):=\log c-\frac{1}{2}\int V(x) \,\beta_{b,c}(dx)
\end{equation}
and
\begin{equation}\label{e:44}
\mathcal{U}_{b,c}(V')> 0 \; \; \;   \text{on a dense subset of}\; \; \; [-2c+b,2c+b].
\end{equation}
The equilibrium measure in this case is $d\mu_{V}=\mathcal{U}_{b,c}(V')d\alpha_{b,c}$.

If this is the case, $(b,c)$ is a solution of 
\begin{equation}
\label{eq:bc}
\begin{cases}
\int cxV'(cx+b)\beta(dx)=2, \\ 
\int V'(cx+b)\beta(dx)=0
\end{cases}
\end{equation}
and 
\begin{equation}\label{e:28}
\begin{split}
E_{V}=-\log c+\int V(x)\beta_{b,c}(dx)-\frac{c^{2}}{2}\iint \left(\frac{V(x)-V(y)}{x-y} \right)^{2}\omega_{b,c}(dx\,dy).
\end{split}
\end{equation}

\end{theorem}

The first part of the theorem is well known and can be seen for example in \cite[Theorems 1.10 and 1.11, Chapter IV]{ST}, while \eqref{e:28} is a combination of \eqref{ep:4} and \eqref{e:rEV}.

For the rest of this paper we will use
the perturbation result for which the following assumptions on the potential $V$ suffice. 
\begin{assumption}\label{A}
\begin{enumerate}
\item $V$ is $C^{3}$.
\item There is a unique maximizer $(c,b)\in(0,\infty)\times \R$ of the function $H$ defined by \eqref{e:31}.
\item  $\mathcal{U}_{b,c}(V')>0$, on $[-2c+b,2c+b]$.  
\end{enumerate}
\end{assumption}  

\begin{remark}
The first two conditions plus \eqref{e:44} are part of the existence of a single interval for the support of the equilibrium measure $\mu_{V}$ as we presented here, while the third assumption is an improved version of \eqref{e:44}.   Moreover, in order to obtain a
Poincar\'e inequality, we must have this third condition satisfied as it was shown in Proposition~\ref{p:2}.  Thus what is written here is just the minimal conditions in order to assure the well posedness of the Poincar\'e inequality.
\end{remark}

Under the conditions of Assumption~\ref{A}, if we perturb the potential $V$ by  $V_{t}=V+tf+t^{2}g+o(t^{2})$ (uniformly on $\R$), where $f,g$  are $C^{3}$ function with all bounded derivative, then $V_{t}$ itself, for small $t$, satisfies the conditions in Assumption~\ref{A}, and thus its equilibrium measure has a one interval support $[-2c_{t}+b_{t},2c_{t}+b_{t}]$ where $c_{t}$ and $b_{t}$ are of $C^{2}$ class in $t$.  

The fact that the support of the equilibrium measure for the perturbed potential is still one interval follows roughly from the fact that the associated $H_{t}$ in Theorem~\ref{t:51} does not change much with $t$ and thus it still has a unique maximum which is close to the one at time $t=0$.  Also the positivity condition \eqref{e:44} with $V$ replaced by $V_{t}$ is satisfied for small $t$. 

The fact that the endpoints of the support of the equilibrium measure, or otherwise stated, $c_{t}$ and $b_{t}$ are $C^{2}$ follows from the implicit function theorem applied to the system \eqref{eq:bc} with $V$ replaced by $V_{t}$.  For a detailed argument on this perturbation,  the reader is referred to the perturbation section in \cite{GP}.

The main result in this section is the following description of  how $E_{V}$ behaves under perturbations.

\begin{theorem}\label{t:3}
Let $V:\R\to\R$ be a potential on $\R$ such that the equilibrium measure $\mu_{V}$ has support $[-2c+b,2c+b]$.   In addition, assume $V_{t}$, $t\in(-\epsilon,\epsilon)$ is a perturbation of $V$ such that 
\[
V_{t}=V+tf(x)+t^{2}g(x)+o(t^{2})
\]
where $f,g:\R\to\R$ are $C^{3}$ on $\R$ with bounded derivatives, and $o(t^{2})$ is uniform on $\R$.   If $E_{t}=E_{V_{t}}$, then 
\begin{equation}\label{e:11}
E_{t}=E_{0}+t\int fd\mu_{V}+t^{2}\bigg( \int gd\mu_{V}-\frac{c^{2}}{2}\iint\left( \frac{f(x)-f(y)}{x-y}  \right)^{2}\omega_{b,c}(dx\,dy)\bigg)+o(t^{2}) .
\end{equation}

\end{theorem}

\begin{proof}
Assume for simplicity (without loss of generality) that $c=1$, $b=0$.  The critical point system \eqref{eq:bc}, reads as
\[
\begin{cases}
\int xV'(x)\,\beta(dx)=2, \\ 
\int V'(x)\,\beta(dx)=0.
\end{cases}
\]
To simplify the writing in this proof,  for any smooth functions $h,k:[-2,2]\to\R$, set 
$ \Pi(h) =\int h\,d\beta $ as in Theorem~\ref {tp:3} and
\begin{equation}\label{eq:L}
\Omega(h,k)=\frac{1}{2} \, \langle\mathcal{N}h,k \rangle
 =\iint \frac{(h(x)-h(y))(k(x)-k(y))}{(x-y)^{2}} \, \omega(dx\,dy). 
\end{equation}
Recast the critical point system in this notation as
\begin{equation}\label{e:12}
\begin{cases}
\Pi(xV')=2, \\ 
\Pi(V')=0.
\end{cases}
\end{equation}
Now, we notice that for small $t$, the equilibrium measure of $V_{t}$ has support $[-2c_{t}+b_{t},2c_{t}+b_{t}]$, where $c_{t}$ and $b_{t}$ depend $C^{2}$ on $t$.  Thus we can write 
\[
c_{t}=1+tc_{1}+t^{2}c_{2}+o(t^{2}),\quad b_{t}=tb_{1}+t^{2}b_{2}+o(t^{2}). 
\]
Continuing, from \eqref{e:28}, 
\[
E_{t}:=-\log c_{t}+\int V_{t}(c_{t}x+b_{t}) \beta(dx) -\frac{1}{2}\iint\left(\frac{V_{t}(c_{t}x+b_{t})-V_{t}(c_{t}y+b_{t})}{x-y} \right)^{2}\omega(dx\,dy).
\]
 Next, a simple Taylor expansion gives 
 \begin{align*}
V_{t}(c_{t}x+b_{t})&=V(c_{t}x+b_{t})+tf(c_{t}x+b_{t})+t^{2}g(c_{t}x+b_{t})+o(t^{2}) \\ 
&=V(x)+t \big (c_{1}x+b_{1}+t(c_{2}x+b_{2}) \big )V'(x)+t^{2} (c_{1}x+b_{1})^{2}V''(x)/2  \\ 
& \quad + tf(x)+t^{2} \big (c_{1}x+b_{1}+t(c_{2}x+b_{2}) \big )f'(x)+t^{2}g(x)+o(t^{2})\\ 
&= V(x)+t \big [(c_{1}x+b_{1})V'(x)+f(x) \big ] \\ 
&\quad+t^{2} \big [(c_{2}x+b_{2})V'(x)+(c_{1}x+b_{1})^{2}V''(x)/2+(c_{1}x+b_{1})f'(x)+g(x) \big ]
      +o(t^{2}).
\end{align*}
Expanding $E_{t}$ to second order yields
\begin{align*}
E_{t}=&E_{0}-t c_{1}-(c_{2}-c_{1}^{2}/2)t^{2}+t c_{1}\Pi(xV')+tb_{1}\Pi(V')+t\Pi(f)\\
& + t^{2} \big [c_{2}\Pi(xV')+b_{2}\Pi(V')+c_{1}^{2}\Pi(x^{2}V'')/2+b_{1}c_{1}\Pi(xV'')
    +b_{1}^{2}\Pi(V'')/2+c_{1}\Pi(xf')+b_{1}\Pi(f')+\Pi(g) \big ] \\ 
& - t \big [c_{1}\Omega(V,xV')+b_{1}\Omega(V,V')+\Omega(V,f) \big ] \\ 
&-t^{2} \big [\Omega((c_{1}x+b_{1})V',(c_{1}x+b_{1})V')/2+\Omega((c_{1}x+b_{1})V',f)+\Omega(f,f)/2 \\ 
&\qquad+\Omega(V,(c_{2}x+b_{2})V')+\Omega(V,(c_{1}x+b_{1})^{2}V'')/2+\Omega(V,(c_{1}x+b_{1})f')+\Omega(V,g) \big ]+o(t^{2}) ,
\end{align*}
and after regrouping the terms according to the power of $t$ it becomes
\begin{equation}\label{e:4}
\begin{split}
E_{t}=& E_{0} +t \big[\Pi(f)-\Omega(V,f)\big]+t^{2}\big[\Pi(g)-\Omega(V,g)-\Omega(f,f)/2\big]  \\
&+t \big [c_{1}(\Pi(xV')-1-\Omega(V,xV'))+b_{1}(\Pi(V')-\Omega(V,V')) \big ] \\ 
&+ t^{2}\big [c_{2}(\Pi(xV')-1-\Omega(V,xV'))+b_{2}(\Pi(V')-\Omega(V,V')) \big ] \\  
&+t^{2} \big [c_{1}^{2}(1+\Pi(x^{2}V'')-\Omega(xV',xV')-\Omega(V,x^{2}V''))+b_{1}^{2}(\Pi(V'')-\Omega(V',V')-\Omega(V,V'')) \\ 
&\qquad +2c_{1}b_{1}(\Pi(xV'')-\Omega(V',xV')-\Omega(V,xV'')) \big ]/2 \\
& +t^{2} \big [c_{1}(\Pi(xf')-\Omega(xV',f)-\Omega(V,xf'))+b_{1}(\Pi(f')-\Omega(V',f)
    -\Omega(V,f')) \big ] +o(t^{2}) .
\end{split}
\end{equation}

Equation \eqref{e:39} gives
\[
\Pi(f)-\Omega(V,f)=\int f d\mu_{V} \quad \text{ and } \quad \Pi(g)-\Omega(V,g)=\int gd\mu_{V}.
\]
and thus the first line of \eqref{e:4} is precisely \eqref{e:11} modulo $o(t^{2})$.   Our remaining task is to prove that the rest of \eqref{e:4} is zero (up to $o(t^{2})$).  

Taking $\phi=\psi=V$ in the second line of \eqref{ep:300} together with \eqref{e:12}, leads to
$ \Omega(V,xV')=1 $
and thus $\Pi(xV')-1-\Omega(V,xV')=0$.   Now using the first line of \eqref{ep:300}
with $\phi=\psi=V$ leads to  $ \Omega(V,V')=0$
which combined with \eqref{e:12}, leads to the conclusion that the second and the third lines of \eqref{e:4} are $0$.  

For the fourth line, take $\phi=\psi=V$ in the second equality of  \eqref{ep:301} plus \eqref{e:12} to conclude that 
\[
\Omega(xV',xV')+\Omega(V,x^{2}V'')=1+\Pi(x^{2}V''). 
\]
Similarly, using the third line of \eqref{ep:301} with $\phi=\psi=V$ in addition to \eqref{e:12}, yields, 
\[
\Omega(V',xV')+\Omega(V,xV'')=\Pi(xV'').
\]
while using the first equality in \eqref{ep:300} for $\phi=V'$ and $\psi=V$ combined with \eqref{e:12}, provides 
\[
\Omega(V',V')+\Omega(V,V'')=\Pi(V'').
\]
These show that the forth and fifth lines of \eqref{e:4} are $0$.  

Finally, using  \eqref{ep:300} for $\phi=V$ and $\psi=f$ together with \eqref{e:12}, yields that 
\begin{equation}\label{epe:1}
\Omega(V',f)+\Omega(V,f')=\Pi(f') \quad \text{ and } \quad \Omega(xV',f)+\Omega(V,xf')=\Pi(xf') 
\end{equation}
which concludes that the last line in \eqref{e:4} is $o(t^{2})$.  This completes the proof.  \qedhere
\end{proof}

\begin{remark}\label{r:100}
Notice that Theorem~\ref{t:3} has a simpler proof in the case the equilibrium measure of $V_{t}$ has a support which is independent of $t$.  Assuming $b=0$, $c=1$,  conform to \eqref{eq:bc}, this amounts to 
\begin{equation}\label{eq:10001}
\begin{cases}\int f'(x)\,\beta(dx)=0, \\ \int xf'(x)\,\beta(dx)=0\end{cases}\quad \text{ and }\quad\begin{cases}\int g'(x)\,\beta(dx)=0,\\ \int xg'(x)\,\beta(dx)=0.\end{cases}
\end{equation}
 The simpler proof alluded to in this case follows directly from the formula \eqref{e:rEV} with $V$ replaced by $V_{t}$ plus expansion  in $t$.  

The content of this theorem says that in fact the same formula holds true even without the constraints from \eqref{eq:10001} but one has to go through a careful examinations of the dependence on the coefficients $c_{1}$, $c_{2}$, $b_{1}$ and $b_{2}$ in \eqref{e:4} and notice that their contributions disappear due to  some remarkable and non-trivial cancellations.  
\end{remark}

Next we study how the equilibrium measure changes under perturbation of the potential.

\begin{theorem}\label{t:4}
Let $V$ satisfy the Assumptions~\ref{A} and let $f\in C^{3}_{b}(\R)$.  Then, in the sense of distributions, 
\[
d\mu_{V+tf}=d\mu_{V}+td\nu_{f}+O(t^{2})
\]
where $\nu_{f}$ is the (unique) signed measure on $[-2c+b,2c+b]$ which solves 
\[
\begin{cases}
2\int\log|x-y|\nu_{f}(dy)=f(x)+C \, \, \, \text{ for almost every } \, x\in[-2c+b,2c+b] , \\
\nu_{f}([-2c+b,2c+b])=0
\end{cases}
\]
where ``almost every'' is with respect to the Lebesgue measure.  
If $b=0$, $c=1$, this can be written in simpler terms as
\begin{equation}\label{e:53}
d\mu_{V+tf}=d\mu_{V}-\frac{t}{2}(\mathcal{N}f)\,d\beta+O(t^{2}).
\end{equation} 
In addition, for $x\in[-2c+b,2c+b]$, 
\begin{equation}\label{e:psi}
\Psi_{f}(x):=\int_{-\infty}^{x}\nu_{f}(dy)=\frac{\sqrt{4c^{2}-(x-b)^{2}}}{2\pi}
   \, (\mathcal{U}_{b,c}f)(x)=\frac{\sqrt{4c^{2}-(x-b)^{2}}}{2\pi}\int \frac{f(x)-f(y)}{x-y}\,d\beta_{b,c}(dy). 
\end{equation}
\end{theorem}

\begin{proof}
As in the proof of Theorem~\ref{t:3}, for small $t$, the equilibrium measure of $V+tf$ has a one interval support $[-2c_{t}+b_{t},2c_{t}+b_{t}]$ and $c_{t},b_{t}$ both depend $C^{3}$ on $t$.   In addition, assuming for simplicity that $c_{0}=1$ and $b_{0}=0$, then we know that for $c_{t}=1+c_{1}t+O(t^{2})$ and $b_{t}=tb_{1}+O(t^{2})$, for some $c_{1},b_{1}\in\R$. 

For a smooth function, $\phi$, using equation \eqref{e:39}, one gets 
\[
\begin{split}
\int \phi d\mu_{V+tf}=\int \phi(c_{t}x+b_{t})\,\beta(dx)&-\iint\frac{(V(c_{t}x+b_{t})-V(c_{t}y+b_{t}))(\phi(c_{t}x+b_{t})-\phi(c_{t}y+b_{t}))}{(x-y)^{2}} \, \omega(dx\,dy)\\
&-  t\iint\frac{(f(c_{t}x+b_{t})-f(c_{t}x+b_{t}))(\phi(c_{t}x+b_{t})-\phi(c_{t}y+b_{t}))}{(x-y)^{2}}
   \, \omega(dx\,dy)
\end{split}
\]
Using Taylor's expansion in $t$, after a little calculation and with the notation from \eqref{eq:L} we continue the above identity with
\[
\begin{split}
\int \phi\, d\mu_{V+tf} &=  \Pi(\phi)-\Omega(V,\phi) \\
&  \quad +t \big [ c_{1}(\Pi(x\phi') -\Omega(xV',\phi)-\Omega(V,x\phi'))
   +b_{1}(\Pi(\phi')-\Omega(V',\phi)-\Omega(V,\phi'))-\Omega(f,\phi) \big ] +O(t^{2}) .
\end{split}
\]
After using \eqref{e:39}, \eqref{e:10} and \eqref{epe:1} the latter can be simplified further into 
\[
\begin{split}
\int \phi d\mu_{V+tf}= & \int \phi\, d\mu_{V}+t\int \phi\, d\nu_{f}+O(t^{2}).
\end{split}
\]
Furthermore,  from \eqref{e:sol}, we have
\[
\nu_{f}(dx)=-\frac{1}{2} \, \mathcal{N}f(x) \,\beta(dx).
\]
Now, \eqref{e:psi} follows from \eqref{ep:302}.  The proof of the theorem is complete.  \qedhere
 
\end{proof}

\section{Poincar\'e's inequality and other functional inequalities}\label{s:fineq}

 This section is devoted to the relationship of the free Poincar\'e inequality
 with the transportation and Log-Sobolev inequalities on the basis of the
 perturbation properties developed in the preceding section. As mentioned
 in the Introduction, the implications from the transportation and Log-Sobolev inequalities
 to the Poincar\'e inequality in the classical case are standard
 (cf. \cite{Ba, OV, BL2, Vi}). Their analogues in the
 free case are surprisingly more involved.

First recall the main functional inequalities to be compared with the free Poincar\'e
inequality (see \cite{LP}).  

\begin{definition} 
\begin{enumerate}
\item The probability measure $\mu_{V}$, or more appropriately, $V$ satisfies a transportation inequality with parameter $\rho>0$, if for any other measure $\mu$, 
\begin{equation}\label{e6:13}
\rho W_{2}^{2}(\nu,\mu_{V})\le E_{V}(\nu)-E_{V}(\mu_{V}),
\end{equation}
where $W_{2}(\nu,\mu)$ is the Wasserstein distance defined as 
\[
W_{2}(\nu,\mu)^{2}=\inf\left\{ \int |x-y|^{2}\pi(dx\,dy) \right\}
\]
where the infimum is taken over all probability measures $\pi$, with marginals $\mu$ and $\nu$ (i.e. $\pi(dx,\R)=\nu(dx)$ and $\pi(\R,dy)=\mu(dy)$). 
In short we refer to the inequality \eqref{e6:13} as $T(\rho)$ which was introduced by Biane and Voiculescu \cite{BV} for the semicircular and in this form by \cite{HPU1}.  

\item Similarly we say that $\mu_{V}$ satisfies a Log-Sobolev, in short  $LSI(\rho)$, $\rho>0$, if for any other (sufficiently nice) probability measure   $\nu$,  
\begin{equation}\label{e:ls}
4\rho (E_{V}(\nu)-E_{V}(\mu_{V}))\le  I_{V}(\nu | \mu_{V})
\end{equation} 
where 
\[
I_{V}(\nu|\mu_{V})=\int (H\nu-V')^{2}d\nu
\]
with 
\begin{equation}\label{e:H}
H\nu(x)=p.v. \int \frac{2}{x-y} \, \nu(dy)
\end{equation} 
taken in the principal value sense.   This inequality was introduced in this form by Biane \cite{Biane2}.

\item At last we say that $\mu_{V}$ satisfies an $HWI(\rho)$, $\rho\in\R$, if for all sufficiently nice probability measure $\nu$,
\begin{equation}\label{e:hwi}
E(\mu)-E(\mu_{V})  \le \sqrt{I_{V}(\mu|\mu_{V})} \, W_{2}(\mu,\mu_{V})-\rho \, W_{2}^{2}(\mu,\mu_{V}).
\end{equation}
\end{enumerate}
\end{definition}

 We should  mention that Log-Sobolev implies transportation
 \cite {L} and that HWI implies Log-Sobolev for $\rho >0$. 
 In particular, although the theorem below provides independent proofs, one main implication
 is the one from the transportation inequality to the Poincar\'e inequality.

A short description of the transportation map is in place here.  For any probability measures, $\mu$, $\nu$ on the real line, with $\mu$ absolutely continuous with respect to the Lebesgue, then  $W_{2}^{2}(\mu,\nu)=\int (\theta(x)-x)^{2}\mu(dx)$ with $\theta$ being the unique non-decreasing transportation map of $\mu$ into $\nu$.  
In addition, if $\mu$ and $\nu$ have densities $g_{\mu}$ and $g_{\nu}$, then 
\begin{equation}\label{e:tran}
\theta'(x) g_{\nu}(\theta(x))=g_{\mu}(x)  \, \, \, \text{ for all } \,  x\in \mathrm{supp}(\mu).  
\end{equation}

Before we proceed to the proof of the main theorem, we want to give a result about the behavior of the transport map of the equilibrium measure of a perturbed potential.   

\begin{proposition}\label{p:4}
Assume that $V$ is a potential satisfying the Assumption~\ref{A} and let $V_{t}=V+tf$, where $f$ is a $C^{3}$ function with all bounded derivatives and let $\mu_{V}$, $\mu_{t}$ be the equilibrium measures of $V$, respectively $V_{t}$.   If $\theta_{t}$ is the transport map from $\mu_{V}$ into $\mu_{t}$, then there is a $C^{1}$ function $\zeta$ on the support of $\mu_{V}$ such that 
\begin{equation}\label{e:l2}
\theta_{t}(x)=x+t\zeta(x)+o(t)
\end{equation}
uniformly in $x$ on the support of $\mu_{V}$.
\end{proposition}

\begin{proof}
By rescaling, we may assume that the support of $\mu_{V}$ is $[-2,2]$.   As we pointed out in the remark following Assumption~\ref{A}, the support of the measure $\mu_{t}$ is $[-2c_{t}+b_{t},2c_{t}+b_{t}]$, where $c_{t}$ and $b_{t}$ are of $C^{2}$ class in $t$.  

From the above presentation of the transportation map, it is clear that $\theta_{t}$ maps $[-2,2]$ into $[-2c_{t}+b_{t},2c_{t}+b_{t}]$ with $\theta_{t}(-2)=-2c_{t}+b_{t}$ and $\theta_{t}(2)=2c_{t}+b_{t}$.   In order to remove the varying endpoints, we rescale $\theta_{t}(x)=2c_{t}\psi_{t}(x)+b_{t}$ and with the help of \eqref{e:tran}, Assumption~\ref{A} and Theorem~\ref{t:51} we learn that 
\begin{equation}\label{e:last}
\psi_{t}'(x)w(t,\psi_{t}(x))\sqrt{4-\psi_{t}^{2}(x)}=w(0,x)\sqrt{4-x^{2}} \text{ for }x\in(-2,2),
\end{equation}
where $w(t,\cdot)$ is the density of the equilibrium measure of $V(c_{t}x+b_{t})$ with respect to the semicircular law.  The important fact to be spelled out here is that $w:[-t_{0}, t_{0}]\times [-2,2]\to (0,\infty)$ is of  $C^{2}$ class for some small enough $t_{0}$.   

Now, if we set $\Psi_{t}(x)=\partial_{t}\psi_{t}(x)$, then 
\begin{equation}\label{e:l20}
\psi_{t}(x)=\psi_{0}(x)+\int_{0}^{t}\Psi_{s}(x)ds=x+t\Psi_{0}(x)+\int_{0}^{t}(\Psi_{s}(x)-\Psi_{0}(x))ds
\end{equation}
and we get the claimed expansion as soon as we prove that $\Psi_{0}$ can be extended to a continuous function on $[-2,2]$ and also that $\sup_{x\in(-2,2)}|\Psi_{s}(x)-\Psi_{0}(x)|$ converges to $0$ when $s$ converges to $0$.

If we take the behavior of the solution $\psi_{t}$ to \eqref{e:last} at points $x\in(-2,2)$, then standard results of perturbation of ordinary differential equations tell us that the perturbation with respect to $t$ is of class $C^{2}$.  However, at the endpoints $\pm2$ this becomes problematic and for this case, one needs a separate analysis.  At least we know that $\partial_{t}\psi_{t}(x)|_{t=0}$ is well defined and uniformly continuous on compact sets of $(-2,2)$.  In particular this justifies the writing of \eqref{e:l2} uniformly on any compact interval in $(-2,2)$ for some continuous function $\zeta$ on $(-2,2)$.

To deal with the behavior at the endpoint $-2$, the other endpoint, $2$,  begin treated similarly.   To this end, we want to remove the square root behavior at $-2$ and for this purpose, we consider $\phi:[0,\infty]\to[-2,2]$, given by 
\[\tag{*}
\phi(u)=\frac{2(u^{2}-1)}{u^{2}+1}.  
\]
Its inverse is $\phi^{-1}(x)=\frac{2+x}{2-x}$ and one of the main reasons of introducing $\phi$ is that  
\[\tag{**}
\sqrt{4-\phi^{2}(u)}=\frac{4u}{1+u^{2}},\text{ and }\phi'(u)=\frac{4u}{(1+u^{2})^{2}}.
\]
Now we take the function $\xi_{t}=\phi^{-1}\circ\psi_{t}\circ\phi$ and hence $\psi_{t}=\phi\circ\xi_{t}\circ\phi^{-1}$ which then gives 
\[
\psi_{t}'(\phi(u))=\phi'(\xi_{t}(u))\xi_{t}'(u)\frac{1}{\phi'(u)}=\xi_{t}'(u)\frac{\xi_{t}(u)}{u}\frac{(1+u^{2})^{2}}{(1+\xi_{t}^{2}(u))^{2}}.
\]
This plugged into \eqref{e:last} with $x=\phi(u)$ yields 
\begin{equation}\label{e:l10}
\xi_{t}'(u)=\frac{u^{2}}{\xi_{t}^{2}(u)}F(t,u,\xi_{t}(u))
\end{equation}
where 
\[
F(t,u,y)=\frac{w(0,\phi(u))(1+y^{2})^{3}}{w(t,\phi(y))(1+u^{2})^{3}}.
\]
Notice that $F$ is a nice positive and $C^{2}$ function in all variables $t,u,y$.  In particular, standard results in ordinary differential equations guarantee that $(t,u)\to\xi_{t}(u)$ is a $C^{2}$ function in both $(t,x)$ on $[-t_{0},t_{0}]\times(0,1]$.

With this $\xi_{t}$ replacing $\psi_{t}$ it suffices to show that the writing \eqref{e:l2} holds true uniformly for $u$ in the interval $[0,1]$.  To do this, set $\eta_{t}(x)=\partial_{t}\xi_{t}(x)$ and write as in \eqref{e:l20}
\[
\xi_{t}(u)=\xi_{0}(u)+t\eta_{0}(u)+\int_{0}^{t}(\eta_{s}(u)-\eta_{0}(u))ds.
\]
Therefore it suffices to prove now that $\eta_{0}$ extends to a continuous function on $[0,1]$ and $\sup_{u\in(0,1]}|\eta_{s}(u)-\eta_{0}(u)|$ converges to $0$ as $s$ goes to $0$.  

We know that $\xi_{t}$ is a continuous function on $[0,1]$ for any small $t$ with $\xi_{t}(0)=0$.  From \eqref{e:l10}, 
\begin{equation}\label{e:l12}
\xi_{t}(u)=\left(\int_{0}^{u}v^{2}F(t,v,\xi_{t}(v))dv \right)^{1/3}
\end{equation}
whose first consequence is that $\xi_{t}(u)/u$ has a limit as $u$ converges to $0$, or otherwise stated that the derivative $\xi_{t}'$ is well defined at $0$ (for any $t\in[-t_{0},t_{0}]$) and in particular can be computed as  
\[
\xi_{t}'(0)=\left( F(t,0,0) \right)^{1/3}.  
\]
What this gives in terms of $\xi_{t}$ is that from \eqref{e:l10} and the $F(0,0,0)=1$, there is a positive constant $C>0$ such that for any $t\in[-t_{0},t_{0}]$ and $u\in[0,1]$, $C^{-1}\le \xi'_{t}(u)<C$ and also $C^{-1}\le \frac{\xi_{t}(u)}{u}\le C$. 

Now we look at our main interest, the derivative $\eta_{t}(u)=\partial_{t}\xi_{t}(u)$.   Observe that from \eqref{e:l10}, it is easy to deduce that 
\[
\eta_{t}'(u)=-\frac{2u^{2}}{\xi_{t}^{3}(u)}\eta_{t}(u)F(t,u,\xi_{t}(u))+\partial_{y}F(t,u,\xi_{t}(u))\eta_{t}(u)+\partial_{t}F(t,u,\xi_{t}(u))\text{ for all }u\in(0,1),t\in[-t_{0},t_{0}],
\]
and on the account of \eqref{e:l10}, we can rewrite this in the form
\[
\eta_{t}'(u)=-\frac{2\xi_{t}'(u)}{\xi_{t}(u)}\eta_{t}(u)+\partial_{y}F(t,u,\xi_{t}(u))\eta_{t}(u)+\partial_{t}F(t,u,\xi_{t}(u))=-\left(\frac{2\xi_{t}'(u)}{\xi_{t}(u)}+a_{t}(u)\right)\eta_{t}(u)+b_{t}(u),
\]
with $a_{t}(u)=\partial_{y}F(t,u,\xi_{t}(u))$ and $b_{t}(u)=\partial_{t}F(t,u,\xi_{t}(u))$.  This implies that there is a constant $L_{t}$ with the property that for $t\in[-t_{0},t_{0}]$ and $u\in(0,1)$,
\[
\eta_{t}(u)\xi^{2}_{t}(u)e^{A_{t}(u)}-\int_{0}^{u}b_{t}(v)\xi^{2}_{t}(v)e^{A_{t}(v)}dv=L_{t}
\]
with $A_{t}(u)=\int_{0}^{u}a_{t}(\sigma)d\sigma$.  Since for any fixed $u>0$, the left hand side is continuous in $t$, it follows that $L_{t}$ is also a continuous function of $t$.  

Setting $B_{t}(u)=\int_{0}^{u}b_{t}(v)\xi^{2}_{t}(v)e^{A_{t}(v)}dv$, we write now,
\[
\xi_{t}(u)=\xi_{0}(u)+\int_{0}^{t}\eta_{s}(u)ds=u+\int_{0}^{t}\frac{L_{s}}{\xi^{2}_{s}(u)e^{A_{s}(u)}}ds+\int_{0}^{t}\frac{B_{s}(u)}{\xi^{2}_{s}(u)e^{A_{s}(u)}}ds,
\]
from which, multiplication by $u^{2}$ and passing to the limit $u\to0$ with the help of \eqref{e:l12} yields 
\[
\int_{0}^{t}\frac{L_{s}}{(F(s,0,0))^{2/3}}ds=0 
\]
for all small $t$, which returns that $L_{t}=0$ for all $t$ small enough.  Hence, it is pretty clear now that 
\[
\eta_{t}(u)=\frac{1}{\xi_{t}^{2}(u)e^{A_{t}(u)}}\int_{0}^{u}b_{t}(v)\xi^{2}_{t}(v)e^{A_{t}(v)}dv
\] 
can be extended to a continuous function at $0$ for each $t\in[-t_{0},t_{0}]$.   In particular $\eta_{0}$ is continuous on the interval $[0,1]$.  In fact a stronger statement holds true here, namely, that 
\[
\sup_{t\in[-t_{0},t_{0}]}|\eta_{t}(u)|\xrightarrow[u\to0]{}0
\]
which follows from the fact that there is a constant $C>0$ such that 
\[
C^{-1}\le\sup_{t\in[-t_{0},t_{0}]}\frac{\xi_{t}(u)}{u}\le C\quad \text{ and }\quad \sup_{t\in[-t_{0},t_{0}],u\in[0,1]}(|a_{t}(u)|+|b_{t}(u)|)\le C, 
\]
which in turn yields that for some $K>0$, 
\begin{equation}\label{e:l21}
\sup_{t\in[-t_{0},t_{0}]}|\eta_{t}(u)|\le K\frac{1}{u^{2}}\int_{0}^{u}v^{2}dv\le Ku/3.  
\end{equation}
What is left to prove here is that $\sup_{u\in(0,1]}|\eta_{s}(u)-\eta_{0}(u)|$ converges to $0$ as $s$ converges to $0$.  If this were not the case, then there would be $\epsilon>0$, $s_{n}\xrightarrow[n\to\infty]{}0$ and $u_{n}\in[0,1]$ such that $|\eta_{s_{n}}(u_{n})-\eta_{0}(u_{n})|\ge\epsilon$.  Without loss of generality we may assume that $u_{n}$ is convergent to some $v\in[0,1]$ and then if $v=0$ contradicts \eqref{e:l21}, while $v\ne0$ contradicts the continuity of $\eta$ at $(0,v)$.  \qedhere
\end{proof}

The following theorem, showing that the free Poincar\'e inequality is
 implied by the transportation or Log-Sobolev inequalitis, is one main conclusion of this work.

\begin{theorem}\label{t:5}
Let $V$ satisfy Assumption~\ref{A} and the support of $\mu_{V}$ be $[-2c+b,2c+b]$.  Then for the measure $\mu_{V}$, and $\rho\ge0$, 
\begin{equation}\label{e:14}
T(\rho)\implies P(\rho) \text{ and }P_{2}(\rho),
\end{equation}
and
\begin{equation}\label{e:15}
LSI(\rho)\implies P(\rho),
\end{equation}
where $P_{2,3,4}(\rho)$ are defined in Theorem~\ref{t:eq} and $P(\rho)$ by \eqref{e1:1}.  
Now, if $\rho\in\R$, then 
\begin{equation}\label{e:58}
HWI(\rho)\implies P_{3}(\rho).
\end{equation}
In particular, if $\rho>0$, then, (cf. Theorem~\ref{t:eq}),  $HWI(\rho)\implies P(\rho)$.

\end{theorem}

\begin{proof} If $\rho=0$, then \eqref{e:14} and \eqref{e:15} are trivial
so that we assume below that $\rho > 0$.  Assume furthermore that $b=0,c=1$.  
We will give here two proofs of \eqref{e:14}.  One is inspired by the classical case and uses the Hamilton-Jacobi semigroup and the dual formulation of the Wasserstein distance, while the other is based directly on the perturbation of the potential.

For the first proof, we employ the tools from the infimum convolution semigroup used in \cite{BL2} for the classical case.   
More precisely, take an arbitrary smooth function $f:[-2,2]\to\R$ and extend it to a smooth compactly supported function on the whole $\R$.  Now use the dual formulation of the Wasserstein distance, which now make the transportation inequality equivalent to 
\[
\rho\left(\int gd\nu -\int fd\mu_{V}\right)\le E_{V}(\nu)-E_{V}(\mu_{V})
\]
for any pair of functions with $g(x)-f(y)\le (x-y)^{2}$.  For  a given $f$, the optimal choice of $g$ is given by $g=Q f$, where 
\begin{equation}\label{e:Q}
(Qf)(x)=\inf_{y\in\R}\{f(y)+(x-y)^{2} \}
\end{equation}
then, 
\[
-\rho \int fd\mu_{V}\le \int (V-\rho Qf)d\nu-\iint \log|x-y|\nu(dx)\nu(dy)-E_{V} 
\]
for any measure $\nu$.  In particular, minimizing over all measures $\nu$, one obtains that 
\[
-\rho \int fd\mu_{V}\le E_{V-\rho Qf}-E_{V}. 
\]
Next, we point out that if we set 
\[
(Q_{t}f)(x)=\inf_{y\in\R}\Big \{f(y)+\frac{(x-y)^{2}}{t} \Big \} ,
\]
then (cf. \cite[Chapter 3]{Evans}) $h(t,x)=(Q_{t}f)(x)$ satisfies the Hamilton-Jacobi equation 
\begin{equation}\label{e:20}
\partial_{t} h+\frac{1}{4}\left(h'\right)^{2}=0.
\end{equation}
Replacing $f$ by $tf$ and using the fact that $Q(tf)=tQ_{t}f=tf-\frac{t^{2}}{4}(f')^{2}+o(t^{2})$ combined with the result of Theorem~\ref{t:3} one is led to 
\begin{equation}\label{e6:1}
 2\rho \iint\left( \frac{f(x)-f(y)}{x-y}  \right)^{2}\omega(dx\,dy)\le\int (f')^{2}d\mu_{V}
\end{equation}
which is exactly $P(\rho)$ from \eqref{e1:1}  for $\mu_{V}$.

Now we turn to the second proof.  We apply the transportation inequality  \eqref{e6:13} with $\mu$ replaced by $\mu_{V+tf}$ and write it as follows
\begin{equation}\label{e:52}
t\int fd\mu_{V+tf}+\rho W_{2}^{2}(\mu_{V+tf},\mu_{V})\le E_{V+tf}-E_{V}.
\end{equation}

Now if $\theta_{t}$ denotes the transportation map of $\mu_{V}$ into $\mu_{V+tf}$,  using Proposition~\ref{p:4} we learn that $\theta_{t}(x)=x+t\zeta(x)+o(t)$, and from here, for any $C^{1}$ function $\phi$ on $[-2,2]$,  
\[
\int \phi(\theta_{t}(x))\mu_{V}(dx)=\int \phi(x)\mu_{V+tf}
\] 
whose expansion in $t$ near $0$ and Theorem~\ref{t:4}, gives 
\[
\int \phi'(x)\zeta(x)\mu_{V}(dx)=-\int \phi(x)\nu_{f}(dx)=-\int \phi'(x)\Psi_{f}(x)dx.
\]
Since $\zeta$ is $C^{1}$, this means that there is a constant $C\in\R$ such that  $\zeta(x)g_{V}(x)=C-\Psi_{f}(x)$ for all $x\in[-2,2]$.  This equality at $x=\pm2$ and the continuity of $\zeta$ at $\pm2$ yields that $C=0$.  Hence 
\[
\zeta(x)=-\frac{\Psi_{f}(x)}{g_{V}(x)}=-\frac{\mathcal{U}f}{\mathcal{U}(V')} \, .
\]
 This means that 
\begin{equation}\label{e8:1}
W_{2}^{2}(\mu_{V+tf},\mu_{V})=t^{2}\int\frac{(\mathcal{U}f)^{2}}{\mathcal{U}(V')} \, d\alpha+o(t^{2}).
\end{equation}
Invoking now \eqref{e:53} and \eqref{e:11}, the result is 
\[
t^{2}\rho \int\frac{(\mathcal{U}f)^{2}}{\mathcal{U}(V')} d\alpha+t^{2}\int fd\nu_{f}\le -\frac{t^{2}}{2}\iint\left( \frac{f(x)-f(y)}{x-y}  \right)^{2}\omega (dx\,dy)+o(t^{2}).
\]
Finally, since (cf. \eqref{e:psi}) 
\[
\int fd\nu_{f}=-\frac{1}{2}\int f\mathcal{N}f\,d\beta=-\iint\left( \frac{f(x)-f(y)}{x-y}  \right)^{2}\omega (dx\,dy),
\]
we arrive at 
\begin{equation}\label{e6:2}
2\rho \int\frac{(\mathcal{U}f)^{2}}{\mathcal{U}(V')} d\alpha \le \iint\left( \frac{f(x)-f(y)}{x-y}  \right)^{2}\omega (dx\,dy)
\end{equation}
which is actually the second equivalent form of $P(\rho)$ from Theorem~\ref{t:eq}.

Now, to prove \eqref{e:15}, we proceed in the same vein.  Take the measure $\mu_{t}$, the equilibrium measure associated to the potential $V+tf$, apply \eqref{e:ls} to it and rewrite it, for small enough $t$, in the following way
\[
4\rho\left(E_{V+tf}-E_{V}-t\int fd\mu_{t}\right)\le t^{2}\int (f')^{2}d\mu_{t} 
\]
where here we used the fact that for small $t$, 
\begin{equation}\label{e8:2}
H\mu_{t}(x)=V'(x)+tf'(x) \text{ for } x \text{ in the support of } \mu_{t}.
\end{equation}  
Therefore, invoking \eqref{e:11} and \eqref{e:53}, we obtain 
\[
\frac{t^{2}}{2}\int f\mathcal{N}f\, d\beta-\frac{t^{2}}{2}\iint \left(\frac{f(x)-f(y)}{x-y} \right)^{2}\omega(dx\,dy)+o(t^{2})\le \frac{t^{2}}{4\rho}\int (f')^{2}\mu_{V}(dx)
\]
which combined with \eqref{ep:9} gives  \eqref{e:15}.

Now, from $HWI(\rho)$, \eqref{e8:1} and \eqref{e8:2}, \eqref{e:58} follows at once. \qedhere
\end{proof}

\begin{remark}  It is interesting to point out that the dual formulation of the transportation implies  $P(\rho)$, while working with the transportation itself (basically the Wasserstein distance) yields $P_{2}(\rho)$ which as we noticed after the proof of Theorem~\ref{t:eq}, is in some sense the dual form of $P(\rho)$.   This is reminiscent of the discussion of Otto-Villani \cite{OV} about the Poincar\'e inequality in the classical case.  
We should also mention that $HWI(\rho)$ for a real $\rho$, gives some sort of ``defective''
version of Poincar\'e. 
\end{remark}

\begin{remark}
We know that the transportation and the Log-Sobolev are satisfied in the case of potentials $V$ which are convex.   
The natural question is to see other cases where these functional inequalities are satisfied.  
As it was pointed out in \cite{Biane2}, there are examples of double well potentials $V$ for which the Log-Sobolev does not hold.   These are cases where the equilibrium measure is supported on two intervals.   It is not clear (at least we do not have any example) if the functional inequalities hold for cases where the measures are supported on several intervals.   
\end{remark}

\begin{remark}
Note that in \cite{OV}, the linearization of classical  $HWI(\rho)$ with $\rho>0$ implies a seemingly stronger inequality than Poincar\'e's with constant $\rho>0$.  Even though Otto and Villani do not point this out, this is in fact equivalent to Poincar\'e's with constant $\rho>0$.  

\end{remark}

\begin{remark}  We pointed out in \cite[Theorem 2]{LP} that if the potential $V$ is such that $V(x)-\rho |x|^{p}$ for some $p>1$, then the following transportation inequality holds 
\begin{equation}\label{eq:t:1}
c_{p} \rho \, W_{p}^{p}(\mu,\mu_{V})\le E(\mu)-E(\mu_{V})
\end{equation}
where $c_{p}=\inf_{x\in\R}\left( |1+x|^{p}-|x|^{p}-p \mathrm{sign}(x)|x|^{p-1}\right)$.  Unfortunately, it turns out that for $1<p<2$,  $c_{p}=0$ and thus this inequality does not say anything.    On the other hand,  for $p>2$ it implies a Poincar\'e's inequality with $\rho=0$.  Indeed, due to the fact that $W_{p}^{p}(\mu_{V+tf},\mu_{V})=o( t^{p})$, for $p>2$ this order is higher than $2$, thus nothing interesting is seen from this inequality as $t$ goes to $0$.
\end{remark} 

The reader might wonder why the classical perturbation argument does not work.   This is what we discuss in the remaining of this section.  

The standard perturbation used in the classical case to linearize the Log-Sobolev or the transportation
inequalities in order to reach the Poincar\'e inequality is $\nu_{t}=(1+tF)\mu_{V}$
for small $t$ and a function $F$ with $\int F\,d\mu_{V}=0$.  We show here that while this gives the free Poincar\'e's for a large class of functions it is not the whole story.  

For simplicity we will assume that $b=0$, $c=1$.  
Take a continuous function $F$ on $[-2,2]$ such that $\int Fd\mu_{V}=0$.  This in particular means that for small $t$, $\nu_{t}=(1+tF)d\mu_{V}$ is again a probability measure.  Thus applying the transportation, we get 
\[
\rho W_{2}^{2}(\nu_{t},\mu_{V})\le t\int VFd\mu_{V}-2t\iint \log|x-y|f\mu(dx)\mu_{V}(dy)-t^{2}\iint \log|x-y|F(x)F(y)\mu_{V}(dx)\mu_{V}(dy)
\]
which, after the use of the fact that $V(x)=2\int \log|x-y|\mu_{V}(dy)+C$ on the support of $\mu_{V}$, leads to 
\begin{equation}\label{ecp:1}
\rho\int (\zeta_{F}(x) )^{2}\mu_{V}(dx)\le-\iint \log|x-y|F(x)F(y)\mu_{V}(dx)\mu_{V}(dy).
\end{equation}
Here in between we used that $\theta_{t}$, the transport map of $\mu_{V}$ into $\mu_{t}$, is given by 
\[
\theta_{t}(x)=x+t\zeta(x)+o(t), 
\]
using essentially the same proof as in Proposition~\ref{p:4}.   Now we proceed as in the second proof of $TCI(\rho)\implies P(\rho)$ from Theorem~\ref{t:5} to deduce that for any $C^{2}$ function on $[-2,2]$,
\[
\int \phi(\theta_{t}(x))\mu_{V}(dx)=\int \phi(x)(1+tF(x))\mu_{V}(dx)
\]
and so expansion in $t$ produces, 
\[
\int \phi'(x)\zeta(x)\mu_{V}(dx)=\int \phi(x)F(x)\mu_{V}(dx)=-\int \phi'(x)G(x)dx
\]
with $G(x)=\int_{-2}^{x}F(y)\mu_{V}(dy)$.  Consequently, $\zeta(x)g_{V}(x)=C+G(x)$, from which at $-2$ and the continuity of $\zeta$, we produce $C=0$, thus, 
\[
\zeta=\zeta_{F}(x)=-\frac{\int_{-2}^{x}F(y)\mu_{V}(dy)}{g_{V}(x)}
\]
where here $g_{V}=\sqrt{(4-x^{2})} \, \mathcal{U}(V')$ is the density of $\mu_{V}$ with respect to the Lebesgue measure.  

In order to make this look like \eqref{e1:4} ($P_{2}(\rho)$), we should take now $F$ such that $Fd\mu_{V}=\nu_{f}=\frac{\mathcal{N}f}{2}\beta$ or equivalently, 
\begin{equation}\label{e6:100}
F(x)=\frac{(\mathcal{N}f)(x)}{(4-x^{2})(\mathcal{U}V')(x)} \, .
\end{equation}
Hence, for those $F$ which can be represented in this form,   the right hand side of \eqref{ecp:1} becomes 
\[
\langle \mathcal{E}\mathcal{N}f,\mathcal{N}f\rangle = \langle f,\mathcal{N}f\rangle
\]
where we used the first equation of Proposition~\ref{p:1}.   Furthermore, now appealing to \eqref{ep:9} and \eqref{ep:302}, it results with 
\begin{equation}\label{e6:102}
2\rho \int \frac{(\mathcal{U}f)^{2}}{\mathcal{U}(V')} \, d\alpha
     \le \iint\left( \frac{f(x)-f(y)}{x-y}  \right)^{2}\omega(dx\,dy),
\end{equation}
which is \eqref{e1:4}.  

However,  in order to make sure that $F$ with the choice \eqref{e6:100} is continuous, we need to guarantee that  $\mathcal{N}f(\pm2)=0$, which otherwise stated (cf. Definition~\ref{d:1}) is the same as
\begin{equation}\label{e6:101}
\int f'(x)\,\beta(dx)=0\quad \text{ and }\quad \int xf'(x)\,\beta(dx)=0.
\end{equation}
This means that we get Poincar\'e's inequality however on a set of functions $f$ satisfying two constraints.   It is not clear to us how to extend \eqref{e6:102} from functions obeying \eqref{e6:101} to any $C^{1}$ function.   

Perhaps a more interesting remark here is that the obstructions from \eqref{e6:101} guarantee that the potential $V_{t}=V+tf$ satisfies, 
\[
\int V_{t}'(x)\,\beta(dx)=0\quad \text{ and }\quad \int xV_{t}'(x)\,\beta(dx)=2.
\] 
These two equations ensure that (cf. \eqref{eq:bc}) the endpoints of the equilibrium measure of $V_{t}$ are $-2$ and $2$, in other words we are just in the situation discussed in Remark~\ref{r:100}.  It seems that in order to overcome this obstruction, a nontrivial argument is needed and this is to some extent the content of Theorem~\ref{t:3} which is also reflected in the different perturbation we used in Section~\ref{s:5}.

A similar argument applies to the implication of free Poincar\'e by the free Log-Sobolev.   

\section*{Acknowledgments}
The second author would like to thank University of Toulouse for its warm and inspiring hospitality where part of this work was carried out.   

We also would like to express true appreciation for the scholar, careful, pertinent and sharp remarks of the anonymous reviewer which transformed the present paper into a better one.

\end{document}